\documentclass[10pt,a4paper,reqno]{amsart} 

\usepackage{longtable}
\usepackage{longtable, multicol}
\usepackage{latexsym}
\usepackage{verbatim}
\usepackage{epsfig}
\usepackage{rotating}
\usepackage{amssymb}
\usepackage[T1]{fontenc}
\usepackage{afterpage}
\usepackage{color}
\usepackage{dsfont}

\usepackage{amssymb,amsfonts,amsmath,stmaryrd}
\usepackage{graphics, graphicx, rotating}
\usepackage{amsthm, amsmath, amssymb}
\usepackage{multirow}

 \def\NW{{\sf NW}}\def\NN{{\sf N}}
 \def\SE{{\sf SE}}\def\EE{{\sf E}}
 \def\SS{{\sf S}}\def\NE{{\sf NE}}
 \def\WW{{\sf W}}\def\SW{{\sf SW}}

\graphicspath{{Plots/}}

\catcode`\@=11
\def\section{\@startsection{section}{1}%
 \z@{.7\linespacing\@plus\linespacing}{.5\linespacing}%
 {\normalfont\bfseries\scshape\centering}}

\def\subsection{\@startsection{subsection}{2}%
  \z@{.5\linespacing\@plus\linespacing}{.5\linespacing}%
  {\normalfont\bfseries\scshape}}

\def\subsubsection{\@startsection{subsubsection}{3}%
 \z@{.5\linespacing\@plus\linespacing}{-.5em}
  {\normalfont\bfseries\itshape}}
\catcode`\@=12

\newtheorem{Theorem}{Theorem}
\newtheorem{Lemma}[Theorem]{Lemma}
\newtheorem{Proposition}[Theorem]{Proposition}

\newtheorem{Corollary}[Theorem]{Corollary}

\def\qed{$\hfill{\vrule height 3pt width 5pt depth 2pt}$}

\newcommand{\ns}{\mathbb{N}}
\newcommand{\zs}{\mathbb{Z}}
\newcommand{\qs}{\mathbb{Q}}

\newcommand{\cs}{\mathbb{C}}

\newcommand{\KL}{\mathbb{K}}
\newcommand{\PL}{\mathbb{P}}

\newcommand{\cS}{\mathcal S}
\newcommand{\cA}{\mathcal A}

\newcommand{\cQ}{\mathcal Q}

\newcommand{\bx}{\bar x}
\newcommand{\by}{\bar y}

\newcommand{\cchi}{\bar \chi}

\newcommand{\beq}{\begin{equation}}
\newcommand{\eeq}{\end{equation}}
\newcommand{\gf}{generating function}
\newcommand{\gfs}{generating functions}

\newcommand{\fps}{formal power series}

\DeclareMathOperator{\comp}{\Theta}
\DeclareMathOperator{\Id}{Id}
\newcommand{\sign}{{\rm sign} }
\newcommand{\sym}{{\rm sym} }
\newcommand{\val}{{\rm val} }

\def\emm#1,{{\em #1}}

\newcommand{\TE}[7]{
#1 & #3
&
 \begin{minipage}{0.8cm}\mbox{}\\

\includegraphics[height=0.8cm]{#2}\\   

\end{minipage}
&\begin{minipage}{6cm}\tiny\mbox{}\\#4\\\hrule\mbox{}\\#5\\\end{minipage} 
& #7 
\\
}

\newcommand{\TsE}[7]{
\TE{#1}{#2}{%
\multirow{2}{*}{%
\begin{minipage}{2.5cm}
\vfill
#3   
\vfill
\end{minipage}}}{#4}{#5}{#6}{#7}
}

\newcommand{\TtE}[7]{
\TE{#1}{#2}{%
\multirow{2}{*}{%
\begin{minipage}{2.5cm}
\vfill
#3   
\vfill
\end{minipage}}}{#4}{#5}{#6}{#7}
}

\newcommand{\tpic}[1]{
\begin{minipage}{0.8cm}
\mbox{}\\\includegraphics[height=0.8cm]{#1}
\end{minipage}} 

\newcommand{\tabB}[5]{
 \StepSet{#2}&$  #3$ & \begin{minipage}{3cm} {\small$  #4$}\end{minipage}&\begin{minipage}{7cm} {\tiny$  #5$}\end{minipage}\\ \hline
}

\newcommand{\StepSet}[1]{
\begin{minipage}{0.5cm}\mbox{}\\
\includegraphics[height=0.3cm]{#1}\\   
\end{minipage}
}

\begin{document}
\title[Walks with small steps in the quarter plane]
{Walks with small steps in the quarter plane}

\author{Mireille Bousquet-M\'elou}
\address{MBM: CNRS, LaBRI, Universit\'e Bordeaux 1, 351 cours de la Lib\'eration,
  33405 Talence Cedex, France}
\email{mireille.bousquet@labri.fr}
\thanks{MBM was supported by the French ``Agence Nationale
de la Recherche'', project SADA ANR-05-BLAN-0372.}
\author{Marni Mishna}
\address{MM: Dept. Mathematics, Simon Fraser University, 8888 University Drive, Burnaby, Canada}
\email{mmishna@sfu.ca}
\thanks{MM was supported by a Canadian NSERC Discovery grant}

\begin{abstract}
Let $\cS\subset \{-1,0,1\}^2\setminus\{(0,0)\}$. We address the enumeration of 
plane lattice walks with steps
in $\cS$,  that start from $(0,0)$ and always remain in the first
quadrant $\{(i,j): i\ge 0, j\ge 0\}$. \emm A priori,, there are   $2^8$
problems of this type, but some  are trivial. Some others are equivalent
to a model of walks confined to a half-plane: such models can be
solved systematically using the kernel method, which leads to
 algebraic \gfs. We focus on the remaining cases, and show that
there are  79  inherently different problems to study. 

To each of
them, we associate a group $G$ of birational transformations. We show
that this group is finite (of order at most 8) in 23 cases, and
infinite in the 56 other cases. 

We present a unified way of solving 22 of the 23 models associated
with a finite group. For all of them, the \gf\ is found to be D-finite. The 23rd
model, known as Gessel's walks, has recently been proved by Bostan
\emm et al., to have an
algebraic (and hence D-finite) solution. 
%
We conjecture that the remaining 56 models, associated with an
infinite group, have a non-D-finite \gf.

Our approach allows us to  recover and refine some known results, and
also to obtain new results. For instance, we prove that walks with \NN, \EE, \WW, \SS,
\SW \ and \NE\ steps have an algebraic  \gf.

\end{abstract}
\maketitle

\date{\today}

\section{Introduction}
The enumeration of lattice walks is a classical topic in
combinatorics. 
Many combinatorial objects (trees, maps, permutations, lattice
polygons, Young tableaux, queues...) can be encoded by lattice walks, so
that lattice path enumeration has many applications.
Given a  lattice, for instance the hypercubic
lattice $\zs^d$, and a finite set of steps $\cS\subset \zs^d$, a
typical problem is to determine how many $n$-step walks with steps
taken from~$\cS$, starting from 
the origin, are confined to a certain region $\cA$ of the space.
If $\cA$ is the whole space, then the length \gf\ of these
walks is a simple rational series. If $\cA$ is a half-space, bounded
by a rational hyperplane, the associated \gf\ is an algebraic
series. Instances of the latter problem have been studied in many
articles since at least the end of the 19th
century~\cite{andre,bertrand}. It is now understood that the \emm
kernel  method, provides a systematic solution to all such
problems, which are, in essence,
one-dimensional~\cite{bousquet-petkovsek-recurrences,banderier-flajolet}.  
Other generic approaches to half-space problems are provided
in~\cite{Duchon98,gessel-factorization}. 

A natural next class of problems is the enumeration of walks
constrained to lie in the intersection of two rational
half-spaces 
--- typically, in the quarter plane.
Thus far, a number of instances have been solved, but no 
unified approach  has emerged yet, and the problem is far from being
completely understood. The generating functions that have been found
demonstrate a more complicated structure than those of half-space problems.
Some examples are algebraic, but for reasons that are poorly
understood
combinatorially~\cite{gessel-proba,Bous05,Mishna-jcta,BoKa08}. Some
examples are, more generally, D-finite, meaning that the \gf\ satisfies a 
linear differential equation with polynomial
coefficients~\cite{guy-bijections,BoPe03,bousquet-versailles,poulalhon-schaeffer}. Some
examples are not  D-finite,  having infinitely many singularities in the complex plane~\cite{BoPe03,Mishna-Rechni}.

We focus in this paper on  walks in the plane confined to the first
quadrant.
The four examples of Figure~\ref{fig:ex}  illustrate the complexity of
this problem:
\begin{itemize}
\item  {\bf Kreweras' walks (steps \WW, \SS, and \NE)}: these were first counted
  in 1965 by Kreweras~\cite{kreweras}. He obtained  a complicated
  expression for  the  number of $n$-step walks ending at $(i,j)$,
  which simplifies drastically when $j=0$. 
It  was then proved by Gessel 
that the associated 3-variable \gf\ is
algebraic~\cite{gessel-proba}. Since then, simpler derivations of this
series have been obtained~\cite{niederhausen-ballot,Bous05}, including
an automated proof~\cite{KaZe07}, and a purely bijective one for walks
ending at the origin~\cite{Bern07}.  
See also~\cite{flatto-hahn,fayolle-livre,Bous05}, where the
stationary distribution of a related Markov chain in the quarter plane
is obtained, and found to be algebraic.
\item {\bf Gessel's walks (steps \EE, \WW, \NE\ and \SW)}: 
Around 2001,
Gessel conjectured  a simple hypergeometric formula for the number of
$n$-step walks ending at the origin. This conjecture was  proved 
 recently by  Kauers, Koutschan and Zeilberger~\cite{KaKoZe08}. Even more
  recently, Bostan and Kauers proved
 that the associated  3-variable
  \gf\ (which counts all quarter plane walks by the length and the coordinates of
  the endpoint) is in
  fact algebraic~\cite{BoKa08}. Strangely enough, the simple numbers
  conjectured by Gessel had not been recognized as the coefficients
  of an algebraic series before.  Both approaches involve, among other
  tools, heavy   computer algebra calculations.
\item {\bf Gouyou-Beauchamps's walks (steps \EE, \WW, \NW\ and \SE)}:
  Gouyou-Beauchamps discovered in 1986 a simple hypergeometric formula for walks
  ending  on the $x$-axis~\cite{gouyou-chemins-montreal}. We
  derive in this paper similar expressions for the total number of
  walks, and for those ending at a prescribed position. The associated
  series are D-finite, but transcendental.
These walks are
related to Young tableaux of height at most 4~\cite{gouyou-tableaux}.  An affine
deformation transforms them into square lattice walks (with \NN, \SS, \EE\ and
\WW\ steps) confined to  the wedge $0\le j \le i$. An
enumeration of these walks involving the number of visits to the
diagonal  appears in~\cite{buks,Nied05,niederhausen05-bis}. 
\item{\bf A non-D-finite case} is provided by walks with \NE, \NW\ and \SE\
  steps. Mishna and Rechnitzer established a complicated expression
  for the \gf\ of these walks, from which they were able to prove that
  this series has infinitely many singularities, and thus cannot be
  D-finite~\cite{Mishna-Rechni}. 
\end{itemize}
Observe that we have defined $\cS$, the set of steps, as a subset of
$\zs^2$ but that we  often use a more intuitive terminology,
referring to $(1,1)$ as a \NE\ step, for instance. 
We occasionally abuse the coordinate notation
 and directly write our steps using $x$'s and $y$'s, writing for example $x\by$
for a \SE\ step.

\medskip
\begin{figure}[hb] 
\center
\hskip -0mm\tpic{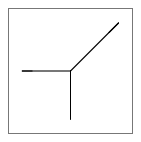}\hskip 35mm\tpic{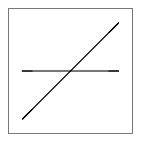}\hskip
35mm\tpic{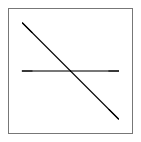}\hskip 30mm \tpic{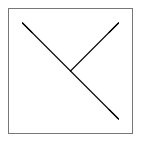}
{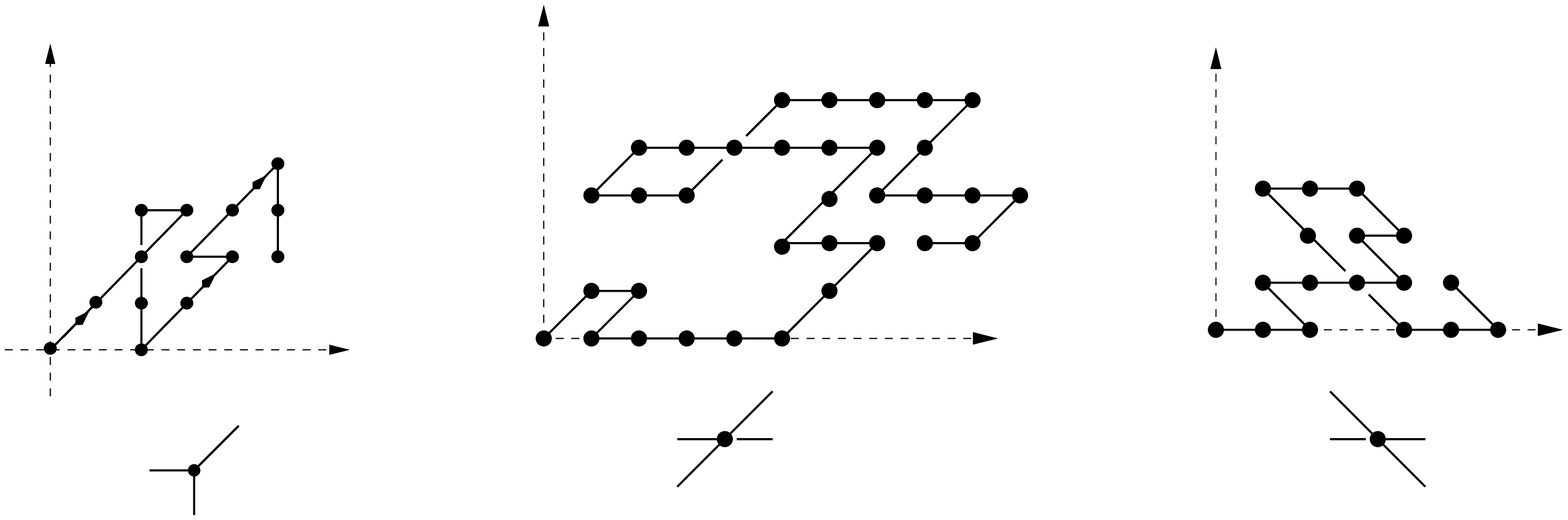}
\vskip 4mm
\caption{Four models of walks in the quarter plane: Kreweras' walks,
  Gessel's walks,  Gouyou-Beauchamps's walks, and the non-D-finite
  example of Mishna and Rechnitzer. The numbers $q(n)$
  count walks of length $n$ confined to the quarter plane that start
  and end at the origin. We have used the notation 
  $
(a)_n=a(a+1)\cdots (a+n-1)$.} 
\label{fig:ex}
\end{figure}

Ideally, we seek generic results or combinatorial conditions
which assure D-finite (or even algebraic) generating functions. 
One  criterion of this type states that if 
the set $\cS$ is invariant by reflection around a vertical
 axis (we  say for short that it has a vertical symmetry) 
and  consists of steps $(i,j)$ such that $|i|\le 1$, 
then the \gf\ of quarter plane walks with steps in $\cS$ is
D-finite~\cite{BoPe03,bousquet-versailles}.  This is also true when
 the set of steps is left invariant by a Weyl group and the walks are
 confined to a corresponding  Weyl chamber~\cite{gessel-zeilberger}.
We are not aware of any other such criteria.

 \subsection{Results}
We restrict our attention to the quarter plane and
\emm small, steps
(\emm i.e.,, the step set $\cS$ is a subset of
$\{-1,0,1\}^2\setminus\{(0,0)\}$). This includes  
the  four  examples above.  We
first  narrow down the $2^8$ possible cases to 79 distinct non-trivial
problems (Section~\ref{sec:models}).  Their step
sets $\cS$ are listed in Tables~\ref{tab:classesD2}
to~\ref{tab:infinite} in Section~\ref{sec:tables}. 
These problems fall 
into two categories, depending on
whether a certain group associated with $\cS$ is finite or
infinite (Section~\ref{sec:group}).
The 23 models associated with a  finite group turn out to be those
that satisfy at least one of the following  properties:
\begin{itemize}
\item[--] the step set possesses a vertical symmetry,
\item[--] the vector sum of the vectors in the step set is 0.
\end{itemize}

In Section~\ref{sec:tools} we develop certain general tools that apply to all
models ---in particular, we explain how to write for each of them
a functional equation that defines the \gf\ of the walks.
Then, we  describe  a \emm  uniform way, to solve this equation for all of
the models associated with a finite group, except one (Gessel's
walks).
The solutions are all D-finite, and even algebraic in three cases
(Sections~\ref{sec:orbit} and~\ref{sec:half-orbit}). 
Note that Gessel's walks are also known to have an algebraic
\gf~\cite{BoKa08}.
We conjecture that all the solutions to models with an infinite group
are non-D-finite. 
We conclude in Section~\ref{sec:questions} with
some comments and  
 questions. 
The tables of Section~\ref{sec:tables} list the 79 models, classified according
to the order of the corresponding group, and provide references
to both the existing 
literature and to the relevant result of this paper.

%
\subsection{Comments and detailed outline of the paper}
The following technical and/or  bibliographical comments may be of interest
to readers who have already worked on similar problems. One will also
find here a more detailed description of the contents of the paper.

The starting point of our approach is a functional equation
 defining the \gf\ $Q(x,y;t)$ that counts quarter plane walks
by the length (variable $t$) and the coordinates of the endpoint
(variables $x$ and $y$). This equation merely reflects a step by step
construction of quarter plane walks. For instance, the equation
obtained for Kreweras' walks (first example in
Figure~\ref{fig:ex}) reads:
$$
\left(1-t(1/x+1/y+xy)\right)Q(x,y;t)=1-t/x\,Q(0,y;t) -t/y\, Q(x,0;t).
$$
Note that there is no obvious way to derive  from the
above identity an equation for, say, $Q(0,0;t)$ or $Q(1,1;t)$. Following Zeilberger's terminology~\cite{zeil-umbral}, we say that the variables $x$ and
$y$ are \emm catalytic,.
One of our objectives is to provide some general principles  that may
be applied to 
any such  linear equation with two catalytic variables. The case of
linear equations 
with \emm one, catalytic variable is well-understood, and the
solutions are always algebraic~\cite{bousquet-petkovsek-recurrences}.

One key tool in our approach is a certain group $G(\cS)$ of birational
transformations that leaves the \emm
kernel, of the functional equation (that is, the coefficient of
$Q(x,y;t)$) unchanged (Section~\ref{sec:group}). We have borrowed this
group from \emm the little yellow book, 
by Fayolle, Iasnogorodski and Malyshev~\cite{fayolle-livre}, 
in which the
authors study the stationary distributions of Markov chains 
with small steps
in the quarter plane. 
Ever since it was imported from probability theory to
combinatorics, this group
 has proved useful (in several disguises, like
the \emm obstinate,, \emm algebraic,, or \emm iterated, kernel method)
to solve various 
enumeration problems, dealing with  walks~\cite{bousquet-versailles,Bous05,BoPe03,Rensburg-Prellberg-Rechni,Mishna-jcta,Mishna-Rechni}, but also with other
objects, like permutations~\cite{bousquet-motifs} or set
partitions~\cite{bousquet-xin,xin-zhang} --- 
the common feature of all these problems being that they boil down to
solving a linear  equation with two catalytic variables.
A striking observation, which
applies to  all  solutions obtained so far,  is that the solution is
D-finite if and only if  the group is finite.
We find that exactly 23 out of our 79 quarter plane models give rise to a
finite group.

We then focus on these 23 models.
For each of them, we derive in Section~\ref{sec:tools}
 an identity between various specializations of $Q(x,y;t)$ which we
 call the \emm orbit sum, (or \emm half-orbit sum,,
when there is an $x/y$ symmetry in $G(\cS)$).

In Section~\ref{sec:orbit}, we show how to derive $Q(x,y;t)$ from the
orbit sum in 19 out of the 20 models that have a finite group and no
$x/y$ symmetry. The number of $n$-step walks ending at $(i,j)$ is
obtained by extracting the coefficient of $x^i y^j t^n$ in a 
rational series  which is easily obtained from the step set and the
group. This implies that the \gf\ $Q(x,y;t)$ is 
D-finite. The form of the solution is reminiscent of a formula
obtained by Gessel and
Zeilberger  for the enumeration of walks confined to a
Weyl chamber, when the set of steps is invariant under the associated
Weyl group~\cite{gessel-zeilberger}.  Indeed,  when the
quarter plane problem happens to be a Weyl chamber problem, our method
can be seen as an algebraic version of the reflection principle
(which is the basis of~\cite{gessel-zeilberger}). 
However, its range of applications seems to be more general.
 We work out in details three cases: walks with
\NN, \WW\ and \SE\ steps (equivalent to Young tableaux with at most 3 rows), walks with \NN,
\SS, \EE, \WW, \SE\ and \NW\ steps (which
do not seem to have been solved before, but
behave very much like the former case), and
finally walks with \EE, \WW, \NW\ and \SE\ steps (studied
in~\cite{gouyou-chemins-montreal}), for which we obtain  new
explicit results.

The results of Section~\ref{sec:half-orbit} may be considered more
surprising:
For the 3 models that have a finite group and an $x/y$
symmetry, we derive the series $Q(x,y;t)$ from the
half-orbit sum and find, remarkably,
that  $Q(x,y;t)$ is always algebraic. We
work out in details all cases: walks with \SS, \WW\ and \NE\ steps
(Kreweras' walks), walks with \EE, \NN\ and \SW\ steps (the reverse steps of
Kreweras' steps) and walks with \NN, \SS, \EE, \WW, \NE\ and \SW\ steps, which, to
our knowledge, have never been studied before.  In particular, we find
that the series $Q(1,1;t)$ that counts walks of the latter
 type, regardless of their endpoint, satisfies a simple quartic equation.

\subsection{Preliminaries and notation}
%
Let $A$ be a commutative ring and $x$ an indeterminate. We denote by
$A[x]$ (resp. $A[[x]]$) the ring of polynomials (resp. \fps) in $x$
with coefficients in $A$. If $A$ is a field, then $A(x)$ denotes the field
of rational functions in $x$, and $A((x))$ the field of Laurent series
in $x$. These notations are generalized to polynomials, fractions
and series in several indeterminates. We 
denote $\bx=1/x$, so that $A[x,\bx]$ is the ring of Laurent
polynomials in $x$ with coefficients in $A$.
The coefficient of $x^n$ in a Laurent  series $F(x)$ is denoted
$[x^n]F(x)$. 
 The \emm valuation, of a Laurent series
$F(x)$ is the smallest $d$ such that $x^d$ occurs in $F(x)$ with a
non-zero coefficient.

The main family of series that we use is that 
 of power series in $t$ with coefficients in
$A[x,\bx]$, that is, series of the form 
$$
F(x;t)=\sum_{n\ge 0, i\in \zs} f(i;n) x^i t^n,
$$
where for all $n$, almost all coefficients $f(i;n)$ are zero. The
\emm positive part, of $F(x;t)$ in $x$ is the following series, which
has coefficients in $x\qs[x]$:
$$
[x^>]F(x;t):=\sum_{n\ge 0, i>0} f(i;n) x^i t^n.
$$
We define similarly the negative, non-negative and non-positive parts of
$F(x;t)$ in $x$, which we denote respectively  by $[x^<]F(x;t), [x^\ge]F(x;t)$ and
$[x^\le]F(x;t)$.

In our generating functions, 
the indeterminate $t$  keeps track of the length of
the walks. We  record the coordinates of the
endpoints with the variables $x$ and $y$. In order to simplify the notation, we  often omit the
dependence of our series in $t$, writing for instance $Q(x,y)$ instead
of $Q(x,y;t)$ for the \gf\ of quarter plane walks.

Recall that a power series $F(x_1, \ldots, x_k) \in \KL[[x_1, \ldots,
    x_k]]$, where $\KL$ is a 
field, is \emm algebraic , (over $\KL(x_1, \ldots, x_k)$) if it satisfies a
non-trivial polynomial equation $P(x_1, \ldots, x_k, F(x_1, \ldots,
x_k))=0$. It is \emm transcendental, if it is not algebraic.
It is \emm D-finite, 
(or \emm holonomic,) 
if the vector space over $\KL(x_1, \ldots, x_k)$
spanned by all partial derivatives of $F(x_1, \ldots, x_k)$ has
finite dimension. 
This means that for all $i\le k$, the series $F$ satisfies a
(non-trivial) linear differential equation in $x_i$ with 
coefficients in $\KL[x_1, \ldots, x_k]$.
We refer to~\cite{lipshitz-diag,lipshitz-df} for a study of these
series. All algebraic series are D-finite.
In Section~\ref{sec:orbit} we use the following result.
\begin{Proposition}\label{prop:pos-part}
 If $F(x,y;t)$ is a rational  power series in $t$, with
  coefficients in $\cs(x)[y,\by]$, then
$[y^>]F(x,y;t)$ is  algebraic over $\cs(x,y,t)$. If the latter series has
  coefficients in $\cs[x,\bx,y]$, its positive part in $x$, that is,
  the series $[x^>][y^>]F(x,y;t)$, is a $3$-variable D-finite series (in $x$, $y$
  and $t$).
\end{Proposition}
\noindent 
The first statement is  a simple adaptation
of~\cite[Thm.~6.1]{gessel-factorization}. The key tool is to expand $F(x,y;t)$
in partial fractions of $y$.   The second statement  relies on  the
fact that the diagonal of a D-finite series is 
D-finite~\cite{lipshitz-diag}. 
One first observes that there is some $k$ such
that $[y^>]F(x,y;tx^k)$ has polynomial coefficients in $x$ and $y$, and
then applies Remark (4), page 377 of~\cite{lipshitz-diag}.   

Below we  also use the fact that a series $F(t)$ with real
coefficients such that $[t^n]F(t)\sim \kappa \mu^n n^{-k}$ with $k
\in \{1,2, 3, \ldots\}$ cannot be algebraic~\cite{flajolet-context-free}.

\section{The number of non-equivalent non-simple models}\label{sec:models}
Since we restrict ourselves to walks with ``small'' steps (sometimes
called steps with small variations), there are only
 a finite number of cases to study, namely $2^8$, the
number of sets $\cS$ formed of small steps. However, some of these
models are trivial (for instance $\cS=\emptyset$, or
$\cS=\{\bx\}$). More generally, it 
sometimes happens that one of the two
constraints imposed by the quarter 
plane holds automatically, at least when the other constraint is satisfied.
Such models are equivalent to problems of walks confined to a
half-space: their  \gf\ is always algebraic and can be derived
automatically using the \emm kernel method,~\cite{bousquet-petkovsek-recurrences,banderier-flajolet}.  
We show  in Section~\ref{sec:alg}
that, out of the $2^8=256$ models, only 138 are truly
2-constraint problems and are thus worth considering in greater
detail.
Then, some of the remaining problems coincide up to 
an $x/y$ symmetry and are thus equivalent.
 As shown in Section~\ref{sec:symm}, one  finally obtains 79 inherently
 different, truly 2-constraint problems. 

\subsection{Easy algebraic cases}\label{sec:alg}
Let us say that  a step $(i,j)$ is $x$-positive if $i>0$. We define
similarly $x$-negative, $y$-positive and $y$-negative steps.
 There are a number  of reasons that may make the enumeration of
quarter plane walks with steps in $\cS$ a simple problem:
\begin{enumerate}
\item If $\cS$ contains no $x$-positive step, we can ignore its
  $x$-negative steps, which will never be used in a quarter
  plane walk: we are thus back to counting walks with vertical steps
  on a (vertical)   half-line. 
 The solution of this problem is always algebraic, and even rational
 if $\cS=\emptyset$ or $\cS=\{y\}$ or  $\cS=\{\by\}$;
\item Symmetrically, if $\cS$ contains no $y$-positive step, the
  problem is simple with an algebraic solution;
\item If $\cS$ contains no $x$-negative step, all walks with steps in
  $\cS$ that start from $(0,0)$ lie in the half-plane $i\ge 0$. Thus
  any walk lying weakly
 above the $x$-axis is automatically a quarter plane walk,
  and the problem boils down to  counting walks confined to 
  the upper half-plane: the corresponding \gf\ is always algebraic;
\item Symmetrically, if $\cS$ contains no $y$-negative step, the
  problem is simple with an algebraic solution.
\end{enumerate}
We can thus restrict our attention to  sets $\cS$ containing 
$x$-positive,  $x$-negative,  $y$-positive  and   $y$-negative steps.
An inclusion-exclusion argument  shows that the number of such sets is
161. More precisely, the polynomial that counts them by cardinality is
\begin{multline*}
P_1(z)=(1+z)^8-4(1+z)^5+2(1+z)^2+4(1+z)^3-4(1+z)+1\\
= 
2\,{z}^{2}+20\,{z}^{3}+50\,{z}^{4}+52\,{z}^{5}+28\,{z}^{6}+8\,{z}^{7}+
{z}^{8}.
\end{multline*}
In the expression of $P_1(z)$, one of the 4 terms $(1+z)^5$ counts sets with no
$x$-positive step, one term $(1+z)^2$ those with no $x$-positive nor
$x$-negative step,  one term $(1+z)^3$ those with no $x$-positive nor
$y$-positive step, and so on.
All the sets $\cS$ we have discarded correspond to problems 
that  either are trivial or can be solved
automatically using the kernel method.

Among the remaining 161 sets $\cS$, some do not contain any step with both
coordinates non-negative: in this case the only quarter plane walk is
the empty walk. These sets are subsets of $\{\bx,\by, x\by,  \bx\by,
 \bx y\}$. But, as we have assumed at this stage that $\cS$
contains $x$-positive and $y$-positive steps, both $x\by$ and $\bx y$
must belong to $\cS$. Hence we exclude $2^3$ of our $161$ step sets,
which leaves us with $153$ sets, the generating polynomial of which is
$$
P_2(z)=P_1(z)-z^2(1+z)^3=
{z}^{2}+17\,{z}^{3}+47\,{z}^{4}+51\,{z}^{5}+28\,{z}^{6}+8\,{z}^{7}+{z}^{8}.
$$

Another, slightly less obvious, source of simplicity of the model is
when one of the quarter plane constraints implies the other. Assume that
all walks with steps in 
$\cS$ that end at a non-negative abscissa automatically end at a
non-negative ordinate (we say, for short, that the  $x$-condition
forces the $y$-condition). This implies in particular that the steps
$\by$ and $x\by$ do not belong to $\cS$. As we have assumed that $\cS$
contains a $y$-negative step, $\bx\by$ must be in $\cS$. But then $x$
cannot belong to $\cS$, otherwise some walks with a non-negative
final abscissa would have a negative final ordinate,  like $x$
followed by $\bx\by$. We
are left with sets $\cS\subset\{\bx, y, xy,\bx y ,\bx\by\}$
containing  $\bx\by$, and also $xy$ (because we need at least one
$x$-positive step). 
Observe that these five steps are those lying above the first diagonal.
Conversely, it is easy to realize that for any such set, the  $x$-condition 
forces the $y$-condition. The generating
polynomial of such \emm super-diagonal, sets is
$z^2(1+z)^3$. 
Symmetrically,  we need not consider \emm sub-diagonal, sets. 
An inclusion-exclusion argument reduces the generating polynomial  of
non-simple cases to 
$$
P_3(z)= P_2(z)-2z^2(1+z)^3+ z^2
=11\,{z}^{3}+41\,{z}^{4}+49\,{z}^{5}+28\,{z}^{6}+8\,{z}^{7}+{z}^{8},
$$
that is to say, to 138 sets $\cS$.

\subsection{Symmetries}\label{sec:symm}

The eight symmetries of the square act on the step sets. However,
only the $x/y$ symmetry (reflection across the first diagonal) leaves
the quarter plane fixed. Thus two step sets obtained from one another
by applying this symmetry lead to equivalent counting problems. As we
 want to count non-equivalent problems,
we need to determine how many among the 138 sets
$\cS$ that are left have the $x/y$ symmetry. We repeat the arguments of
the previous subsection, counting
 only symmetric models. We successively obtain
$$
P_1^{\sym}= (1+z)^2 (1+z^2)^3-2(1+z)(1+z^2)+1,
$$
$$
P_2^{\sym}=P_1^{\sym}-z^2(1+z)(1+z^2),
$$
$$
P_3^{\sym}=P_2^{\sym}-z^2= 3\,{z}^{3}+
5\,{z}^{4}+5\,{z}^{5}+4\,{z}^{6}+2\,{z}^{7}+{z}^{8}.
$$
For instance, the term we subtract from $P_1^{\sym}$ to obtain
$P_2^{\sym}$ counts symmetric subsets of $\{\bx,\by, x\by,  \bx\by,
 \bx y\}$ containing  $x\by$ and $\bx y$.
The generating polynomial of (inherently different) models that are neither
trivial, nor equivalent to a 1-constraint problem is thus
$$
\frac 1 2 \left(P_3+P_3^{\sym}\right)=
7\,{z}^{3}+23\,{z}^{4}+27\,{z}^{5}+16\,{z}^{6}+5\,{z}^{7}+{z}^{8}.
$$
This gives a total of 79 models, shown in Tables~\ref{tab:classesD2}
to~\ref{tab:infinite}.

\section{The group of the walk}\label{sec:group}
%
Let $\cS$ be a set of small steps containing 
 $x$-positive, $x$-negative,  $y$-positive and  $y$-negative steps.
 This includes  the 79 sets we wish to study.
Let $S(x,y)$ denote the generating polynomial
of the steps of $\cS$: 
\beq\label{S-def}
S(x,y)=\sum_{(i,j)\in \cS} x^iy^j.
\eeq
It is a Laurent polynomial in $x$ and $y$. Recall that $\bx$ stands
for $1/x$, and $\by$ for $1/y$.  
Let us write
\beq\label{A-B}
S(x,y)= A_{-1}(x) \by + A_0(x) + A_1(x) y
= B_{-1}(y) \bx + B_0(y) + B_1(y) x.
\eeq
By assumption, $A_1, B_1, A_{-1}$ and $B_{-1}$ are non-zero.
Clearly, $S(x,y)$ is left unchanged by the following rational
transformations:
$$
\Phi: (x,y)\mapsto \left(\bx \frac {B_{-1}(y)}{B_{1}(y)}, y\right)
\quad \hbox{and} \quad 
\Psi: (x,y)\mapsto \left(x, \by \frac {A_{-1}(x)}{A_{1}(x)}\right).
$$
Note that both $\Phi$ and $\Psi$ are involutions, and thus \emm
bi,\,rational transformations. By composition, they
generate a group 
 that we  denote $G(\cS)$, or $G$ if
there is no risk of confusion. This group is isomorphic to a dihedral
group $D_n$ of order $2n$, with $n\in\ns\cup \{\infty\}$. For each $g\in G$, one has
$S(g(x,y))=S(x,y)$. The \emm sign, of $g$ is 1 (resp.~$-1$) if $g$ is
the product of an even (resp. odd) number of generators $\Phi$ and
$\Psi$. 

\medskip\noindent
{\bf Examples}\\
\noindent 
{\bf 1.}
Assume $\cS$ is left unchanged by a
reflection across a vertical line. This is equivalent to saying that
$S(x,y)=S(\bx,y)$, or  that $B_1(y)=B_{-1}(y)$, or that
$A_i(x)=A_i(\bx)$ for $i=-1,0, 1$. Then 
 the orbit of $(x,y)$ under the action of $G$ reads
$$
(x,y)  
 {\overset{\Phi}{\longleftrightarrow}} (\bx,y)
 {\overset{\Psi}{\longleftrightarrow}} (\bx,C(x)\by)
 {\overset{\Phi}{\longleftrightarrow}} (x,C(x)\by)
 {\overset{\Psi}{\longleftrightarrow}} (x,y),
$$
with $C(x)=\frac{A_{-1}(x)}{A_{1}(x)}$, so that $G$ is finite of order
4.

Note that there may exist rational transformations on $(x,y)$ that
leave $S(x,y)$ unchanged but are
\emm not, in $G$. For instance, if
$\cS=\{\NN,\SS,\EE,\WW\}$, the map $(x,y)\mapsto (y,x)$ leaves
$S(x,y)$ unchanged, but the orbit of $(x,y)$ under $G$ is $\{(x,y),
(\bx, y), (\bx, \by), (x, \by)\}$.

\noindent 
{\bf 2.}
Consider the case $\cS=\{\bx,y,x\by \}$. We have $A_{-1}(x)=x$,
$A_1(x)=1$, $B_{-1}(y)=1$,  $B_1(y)=\by$. The   transformations are
$$
\Phi: (x,y)\mapsto (\bx y,y) \quad \hbox{and}   \quad \Psi: (x,y)\mapsto(x,x\by),
$$
and they generate a group of order 6:
\beq\label{orbit-tandem}
(x,y)  
 {\overset{\Phi}{\longleftrightarrow}} (\bx y,y)
 {\overset{\Psi}{\longleftrightarrow}} (\bx y,\bx)
 {\overset{\Phi}{\longleftrightarrow}} (\by,\bx)
 {\overset{\Psi}{\longleftrightarrow}} (\by,x\by)
 {\overset{\Phi}{\longleftrightarrow}} (x,x\by)
 {\overset{\Psi}{\longleftrightarrow}} (x,y).
\eeq

\noindent 
{\bf 3.}
Consider now the case $\cS=\{\bx,\by,xy \}$, which differs from the
previous one by a rotation of 90 degrees. We have $A_{-1}(x)=1$,
$A_1(x)=x$, $B_{-1}(y)=1$,  $B_1(y)=y$. The two   transformations are
$$
\Phi: (x,y)\mapsto (\bx \by,y) \quad \hbox{and}   \quad \Psi: (x,y)\mapsto(x,\bx\by),
$$
and they also generate a group of order 6:
$$
(x,y)  
 {\overset{\Phi}{\longleftrightarrow}} (\bx \by,y)
 {\overset{\Psi}{\longleftrightarrow}} (\bx \by,x)
 {\overset{\Phi}{\longleftrightarrow}} (y,x)
 {\overset{\Psi}{\longleftrightarrow}} (y,\bx\by)
 {\overset{\Phi}{\longleftrightarrow}} (x,\bx\by)
 {\overset{\Psi}{\longleftrightarrow}} (x,y).
$$ 
As shown by the following lemma, this is not a coincidence.

\begin{Lemma}\label{lem:sym}
  Let $\cS$ and $\tilde \cS$ be two sets of steps differing by one of
  the $8$ symmetries of the square. Then the groups $G(\cS)$ and
  $G(\tilde \cS)$ are isomorphic.
\end{Lemma}
\begin{proof}
The group of symmetries of the square is generated by the two
reflections $\Delta$ (across the first diagonal) and $V$ (across a
vertical line). Hence it suffices to prove the lemma when $\tilde S=
\Delta(S)$ and  when $\tilde S=V(S)$.
 We denote by $\Phi $ and $\Psi$ the transformations associated with
 $\cS$, and by $\tilde \Phi$ and $\tilde \Psi$ those associated with
 $\tilde \cS$.

Assume $\tilde S=\Delta(S)$. We have $\tilde A_i(x)=B_i(x)$ and
$\tilde B_i(y)=A_i(y)$. Denote by $\delta$ the involution
that swaps the coordinates of a pair: $\delta(x,y)=(y,x)$. An
elementary calculation gives
$$
\tilde \Phi= \delta \circ \Psi\circ \delta 
\quad \hbox{and} \quad
\tilde \Psi= \delta \circ \Phi\circ \delta 
$$
so that the groups  $G(\cS)$ and  $G(\tilde \cS)$ are conjugate by
$\delta$.

Assume now $\tilde S=V(S)$. We have $\tilde A_i(x)=A_i(\bx)$ and
$\tilde B_i(y)=B_{-i}(y)$. Denote by $v$ the involution
that replaces  the first coordinate of a pair by its reciprocal:
$v(x,y)=(\bx,y)$. An elementary calculation gives
$$
\tilde \Phi= v \circ \Phi\circ v 
\quad \hbox{and} \quad
\tilde \Psi= v \circ \Psi\circ v 
$$
so that the groups  $G(\cS)$ and  $G(\tilde \cS)$ are conjugate by
$v$.
\end{proof}

\begin{Theorem}\label{thm:infinite}
  Out of the $79$ models under consideration, exactly
 $23$ are  associated  with a finite group: 
\begin{enumerate}
\item[--] $16$ have a vertical symmetry and thus  a group of  order $4$,
\item[--] $5$ have a group of order $6$, 
\item[--] $2$ have  a group of order $8$. 
\end{enumerate}
 \end{Theorem}
\begin{proof}
Given that $\Phi$ and $\Psi$ are involutions,  the group they generate
is finite of order $2n$ if and only if $\comp:=\Psi\circ \Phi$ has finite
order $n$. It is thus easy to prove that one of the groups $G(\cS)$ has
order $2n$: one computes  the $m^{\hbox{\small th}}$ iterate $\comp^{m}$
for $1\le m \le n$, and 
checks that only the last of these transformations is the identity. 

We have already seen in the examples above that models with a vertical symmetry
have a group of 
order 4. We leave it to the reader  to check that the models of
Tables~\ref{tab:classesD3} 
and~\ref{tab:classesD4} have groups of order 6 and 8, respectively.
These tables 
give the orbit  of $(x,y)$ under the action of $G$, the elements being
listed in the following order: $(x,y)$, $\Phi(x,y)$, $\Psi\circ
\Phi(x,y)$, and so on.  

Proving that one of the groups $G(\cS)$ is infinite is a more difficult
task.
We apply two different strategies, depending on $\cS$. The first one uses valuations and
works for the   five step sets  of Figure~\ref{tab:5classes-infinite}.
 These are the  sets of our collection 
for which all elements $(i,j)$ satisfy $i+j \ge 0$. We are very
grateful to  Jason Bell, who suggested to us a second strategy which
turned out to apply to the remaining cases. 

\begin{figure}[h] \center
\hfill\tpic{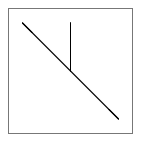}\hskip 10mm
\tpic{01010001}\hskip 10mm\tpic{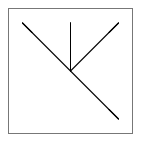}
\hskip 10mm\tpic{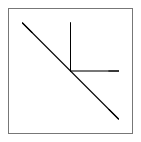}\hskip 10mm\tpic{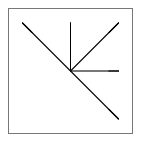}
\hfill\mbox{}
\caption{Five step sets with an infinite group.}
\label{tab:5classes-infinite}
\end{figure}

\noindent {\bf 1. The valuation argument} \\
 Let $z$ be an indeterminate, and let
$x$ and $y$ be Laurent series in $z$ with coefficients in $\qs$,
of respective valuations $a$ and $b$. We assume  that the \emm
trailing, coefficients of these series, namely $[z^a]x$ and $[z^b]y$,
are positive. 
Let us define $x'$ by  $\Phi(x,y)=(x',y)$.
Then  the trailing coefficient of $x'$ (and $y$) is positive,
and the  valuation of $x'$ (and $y$) only depends on $a$ and $b$:
$$
\phi(a,b):=\left(\val(x'), \val(y)\right)= 
\left\{\begin{array}{ll}
  \left( -a+ b( v_{-1}^{(y)}-v_{1}^{(y)}), b\right) &
  \hbox{if } b\ge 0,\\
  \left( -a+ b( d_{-1}^{(y)}-d_{1}^{(y)}), b\right) &
  \hbox{if } b\le 0,
\end{array}\right.
$$
where $v_{i}^{(y)}$ (resp. $d_{i}^{(y)}$) denotes the valuation
(resp. degree) in $y$ of $B_i(y)$, for $i=\pm 1$. 
Similarly, $\Psi(x,y):=(x,y')$ is
well-defined, and the  valuations of $x$ and $y'$ only depend on $a$
and $b$:
$$
\psi(a,b):=\left(\val(x), \val(y')\right)= 
\left\{\begin{array}{ll}
  \left(a, -b+ a( v_{-1}^{(x)}-v_{1}^{(x)})\right) &
  \hbox{if } a\ge 0,\\
  \left(a, -b+ a( d_{-1}^{(x)}-d_{1}^{(x)})\right) &
  \hbox{if } a\le 0,
\end{array}\right.
$$
where $v_{i}^{(x)}$ (resp. $d_{i}^{(x)}$) denotes the valuation
(resp. degree) in $x$ of $A_i(x)$, for $i=\pm 1$.

In order to prove that $G$ is infinite, it suffices to prove that the
group $G'$ generated  by $\phi$ and $\psi$ is infinite. To prove the latter
statement, it suffices to exhibit  $(a,b)\in \zs^2$ such that the
orbit of $(a,b)$ under the action of $G'$ is infinite.

Let $\cS$ be one of the five sets of
Figure~\ref{tab:5classes-infinite}. Then $A_{-1}(x)=x$ 
and $B_{-1}(y)=y$, so that $v_{-1}^{(x)}= d_{-1}^{(x)}=v_{-1}^{(y)}=
d_{-1}^{(y)}=1$. Also, $v_{1}^{(x)}=v_{1}^{(y)}=-1$ as $\cS$ contains
the steps $\bx y $ and $x\by$.
Hence the transformations $\phi$ and $\psi$ read:
$$
\phi(a,b)=\left\{
\begin{array}{ll}
  \left( -a+2 b, b\right) &
  \hbox{if } b\ge 0,\\
  \left( -a +b(1-d_1^{(y)}), b\right) &
  \hbox{if } b\le 0,
\end{array}\right.
\quad \quad 
\psi(a,b)=\left\{
\begin{array}{ll}
  \left(a, 2a-b\right) &
  \hbox{if } a\ge 0,\\
  \left(a, -b+a(1-d_1^{(x)}) \right) &
  \hbox{if } a\le 0.
\end{array}\right.
$$
It is easy to check, by induction on $n\ge 0$, that
$$
(\psi\circ\phi)^n(1,2)= (2n+1, 2n+2) \quad \hbox{and}\quad 
\phi(\psi\circ\phi)^n(1,2)= (2n+3, 2n+2).
$$
(All these pairs have positive entries, so that we never need to know
$d_1^{(y)}$ or $d_1^{(x)}$.)
Hence the orbit of $(1,2)$ under the action of $\phi$ and $\psi$ is
infinite, and so are the groups $G'$ and $G$.

We believe, from our computer experiments, that the groups $G'$
associated with the remaining 51 models of
Table~\ref{tab:infinite} are finite, 
and hence, cannot be used to prove that $G$ is infinite. Instead, we
use for these models a different argument based on the fixed points of
$\comp=\Psi\circ\Phi$.

\medskip
\noindent {\bf 2. The fixed point argument}\\
We are  left with 51 models.  Thanks to Lemma~\ref{lem:sym}, we only 
need to prove that (roughly) a  
quarter of them are associated with a finite group: if  $G(\cS)$
is infinite, then $G(\tilde \cS)$ is infinite for all sets $\tilde \cS$ that
differ from $\cS$ by a symmetry of the square\footnote{Why a quarter,
rather than an eighth? Recall that, if two (distinct) models differ by
an $x/y$ symmetry, only one of them appears in Table~\ref{tab:infinite}.}.
This leaves 14 models to study, listed in Table~\ref{tab:ffpt}. 

Assume $\comp=\Psi\circ\Phi$ is well-defined in the neighborhood of $(a,b) \in
\cs^2$, and that this point is fixed by $\comp$. Note that  $a$
and $b$ are algebraic over $\qs$. Let us write
$\comp=(\comp_1, \comp_2)$, where $\comp_1$ and  $\comp_2$ are the two
coordinates of $\comp$. Each $\comp_i$ sends the pair $(x,y)$ to a rational
function of $x$ and $y$. The local expansion
of $\comp$ around $(a,b)$  reads
$$
\comp(a+u,b+v)= (a,b) + (u,v)J_{a,b} + O(u^2) + O(v^2) +O(uv),
$$
where $J_{a,b}$ is the Jacobian matrix of $\comp$ at $(a,b)$:
$$
J_{a,b} = \left(
\begin{array}{ll}
  \displaystyle \frac{\partial \comp_1}{\partial x}(a,b) & \displaystyle \frac{\partial
    \comp_2}{\partial x}(a,b) \\
\\
\displaystyle   \frac{\partial \comp_1}{\partial y}(a,b) & \displaystyle \frac{\partial
    \comp_2}{\partial y}(a,b)
\end{array}
\right).
$$
Iterating the above expansion gives, for $m\ge 1$,
$$
\comp^{m}(a+u,b+v)= (a,b) + (u,v)J_{a,b}^m + O(u^2) + O(v^2) +O(uv).
$$
Assume $G(\cS)$ is finite of order $2n$, so that $\comp$ has 
order $n$. Then $\comp^{n}(a+u,b+v)= (a,b)+(u,v)$, and the above
equation shows that $J_{a,b}^n$ is the identity matrix. In particular, all
 eigenvalues of $J_{a,b}$ are roots of unity.

This gives us a strategy
 for proving that a group $G(\cS)$ is infinite:
find a fixed point $(a,b)$ for $\comp$, and compute the characteristic
polynomial $\chi(X)$ of the Jacobian matrix $J_{a,b}$. This is a
polynomial in $X$ with coefficients in $\qs(a,b)$.  
In order to decide whether the roots of $\chi$ are roots of unity, we
eliminate $a$ and 
$b$ (which are algebraic numbers) from the equation $\chi(X)=0$ to
obtain a polynomial  $\cchi(X) \in\qs[X]$ that vanishes for all eigenvalues of $J_{a,b}$: if none of its
factors is cyclotomic, we can conclude that $G(\cS)$ is infinite. As
all cyclotomic polynomials of given degree are known, this procedure
is effective.

Let us treat one case in detail, say $\cS=\{x,y,\by, \bx\by\}$ 
(the first case in Table~\ref{tab:ffpt}).
 We have
$$
\comp(x,y) =\Psi\circ\Phi(x,y)=\left( \bx \by, x+\by\right).
$$
Every pair $(a,b)$ such that $a^4+a^3=1$ and $b=1/a^2$ is 
fixed by $\comp$. Let us choose one such pair.  The Jacobian matrix reads
$$
J_{a,b}= \left(
\begin{array}{cc}
  -1 & 1
\\
-a^3 & -a^4
\end{array}
\right),
$$
and its characteristic polynomial  is 
$$
\chi(X):=\det( X \Id - J_{a,b}) = {X}^{2}+X(1+{a}^{4})+{a}^{3}+{a}^{4}.
$$
Let $X$ be a root of this polynomial. By eliminating $a$ (which
satisfies $a^4+a^3=1$), we obtain
$$
\cchi(X):={X}^{8}+9\,{X}^{7}+31\,{X}^{6}+62\,{X}^{5}+77\,{X}^{4}+62\,{X}^{3}+31
\,{X}^{2}+9\,X+1
=0.
$$
This polynomial is irreducible, and distinct from all cyclotomic
polynomials of degree 8. Hence none of its roots  are  roots of unity, no
power of $J_{a,b}$ is equal to the identity matrix, and the group
$G(\cS)$ is infinite.

This strategy turns out to work for  all  models of
Table~\ref{tab:ffpt}. This table gives, for each model, the
algebraic equations defining the fixed point
$(a,b)$ that we choose 
(for instance, the ``condition'' $a^4+a^3-1$
occurring on the first line means that $a^4+a^3-1=0$), 
and a polynomial  $\cchi(X) \in\qs[X]$ that vanishes at all eigenvalues of the
Jacobian matrix, in factored form. One then checks that no factor of
this polynomial is cyclotomic.
 
It may be worth noting that this second strategy
 does \emm not, work for the five models of
 Figure~\ref{tab:5classes-infinite}: in the first three cases,
 $\comp$ has no fixed  point; 
in the last two cases, it has a fixed
 point $(a,b)$, but the sixth power of the Jacobian
 matrix $J_{a,b}$ is the identity
(of course, $\Theta^6$ is \emm not, the identity; more precisely, the
 expansion of $\Theta^6(a+u,b+v)$ involves cubic terms in $u$ and $v$).
 \end{proof}

\noindent {{\bf Remark.}} To put this discussion in a larger
framework, let us mention that 
the group of birational transformations or $\cs^2$
(or of the projective plane $\PL^2(\cs)$), called the \emm Cremona
group,, 
has been the object of many studies in algebraic geometry since the end of the
19th century~\cite{kantor,wiman}. It seems possible
 that, given the
attention already paid to the classification of finite subgroups of
this group (see e.g.~\cite{blanc,bayle-beauville}), there is a
generic or automatic test for finiteness that can be applied to all our examples. 
%

\section{General tools}\label{sec:tools}

Let $\cS$ be one of the 79 step sets of
Tables~\ref{tab:classesD2} to \ref{tab:infinite}. Let $\cQ$ be
the set of walks that start from $(0,0)$, take their steps  from $\cS$
and always remain in the first quadrant. Let $q(i,j;n)$ be the number
of such walks that have length $n$ and end at position $(i,j)$. Denote
by $Q(x,y;t) \equiv Q(x,y)$ the associated \gf:
$$
Q(x,y;t)=\sum_{i,j,n\ge 0} q(i,j;n) x^i y^j t^n.
$$
It is a \fps\ in $t$ with coefficients in $\qs[x,y]$. 

\subsection{A functional equation} 

\begin{Lemma}\label{lem:eq-func}
As a power series in $t$, the 
 \gf\ $Q(x,y)\equiv Q(x,y;t)$ of walks with steps taken from~$\cS$ starting from $(0,0)$ and staying in the first quadrant is
characterized by the following functional equation: 
  $$
K(x,y)xyQ(x,y)= xy- txA_{-1}(x) Q(x,0) -tyB_{-1}(y) Q(0,y)+ t\epsilon Q(0,0)
$$
where
$$
K(x,y)= 1-tS(x,y)=1-t \sum_{(i,j)\in \cS } x^iy^j
$$
is called the \emm kernel, of the equation,  the polynomials
$A_{-1}(x)$ and $B_{-1}(y)$ are the coefficients of $\by$ and $\bx$ in
$S(x,y)$, as
described by~\eqref{A-B}, and  $\epsilon$ is $1$ if $(-1,-1)$
is one of the allowed steps, and $0$ otherwise.
\end{Lemma}
\begin{proof}
  We construct walks step by step, starting from the empty walk and
  concatenating a new step at the end of the walk at each
  stage. The empty walk has weight 1. The \gf\ of walks obtained by
  adding a step of $\cS$ at the end of a walk of $\cQ$ is
  $tS(x,y)Q(x,y)$. However, some of these walks exit the quadrant:
  those obtained by concatenating a $y$-negative step
to a walk ending at ordinate $0$, and  those obtained by concatenating
an $x$-negative step
to a walk ending at abscissa $0$. Walks ending at
  ordinate (resp. abscissa) $0$ are 
  counted by the series $Q(x,0)$ (resp. $Q(0,y)$). Hence we must
  subtract the series $t\by A_{-1}(x) Q(x,0)$ and $t\bx B_{-1}(y)
  Q(0,y)$. However, if $(-1,-1) \in \cS$, we have subtracted  twice the series counting
  walks obtained by concatenating this step to a walk ending at
  $(0,0)$: we must thus add the series $\epsilon t\bx\by Q(0,0)$. This
  inclusion-exclusion argument gives
$$
Q(x,y)=1+ tS(x,y)Q(x,y)-t\by A_{-1}(x) Q(x,0)-t\bx
B_{-1}(y)Q(0,y)+\epsilon t\bx\by Q(0,0),
$$
which, multiplied by $xy$, gives the equation of the lemma.

The fact that it characterizes $Q(x,y;t)$ completely (as a power
series in $t$) comes from the fact that the coefficient of $t^n$ in
$Q(x,y;t)$ can be computed inductively using this equation. This is of
course closely related to the fact that we have used a  recursive
description of walks in $\cQ$ to obtain the equation.
\end{proof}

\subsection{Orbit sums}
We have seen in Section~\ref{sec:group} that all transformations $g$
of the group $G$ associated with the step set $\cS$ leave the  polynomial
$S(x,y)$ unchanged. Hence they also leave the kernel
$K(x,y)=1-tS(x,y)$ unchanged. Write  the equation of Lemma~\ref{lem:eq-func} as
$$
K(x,y)xyQ(x,y)=xy-F(x)-G(y)+t\epsilon Q(0,0),
$$
with $F(x)=tx A_{-1}(x) Q(x,0)$ and $G(y)=tyB_{-1}(y)Q(0,y)$. 
Replacing $(x,y)$ by $\Phi(x,y)=(x',y)$ gives
$$
K(x,y)x'yQ(x',y)=x'y-F(x')-G(y)+t\epsilon Q(0,0).
$$
The difference between the former and latter  identities reads:
$$
K(x,y)\big( xyQ(x,y)-x'yQ(x',y)\big)= xy-x'y-F(x)+F(x').
$$
The term $G(y)$ has disappeared. We can repeat this process, and
add to this identity the functional equation of Lemma~\ref{lem:eq-func}, evaluated at
$(x',y')=\Psi(x',y)$. This gives:
$$
K(x,y)\big( xyQ(x,y)-x'yQ(x',y)+x'y'Q(x',y')\big)= xy-x'y+x'y'-F(x)-G(y')+t\epsilon Q(0,0).
$$
Now the term $F(x')$ has disappeared. If $G$ is finite of order $2n$,
we can repeat the procedure until we come back to
$(\Psi\circ\Phi)^n(x,y)=(x,y)$.  That is to say, we 
form the alternating sum of the equations over the orbit of $(x,y)$.
All unknown functions on the right-hand side finally vanish, giving
the following proposition, where we use the notation
$$
\hbox{for } g \in G, \quad g(A(x,y)):=A(g(x,y)).
$$

\begin{Proposition}[{\bf Orbit sums}]\label{prop:orbit-sum}
Assume the group $G(\cS)$ is finite. Then
  \beq\label{eq:orbit}
\sum_{g\in G} \sign (g) g(xyQ(x,y;t))= \frac 1 {K(x,y;t)} \sum_{g\in G}
\sign (g) g(xy).
\eeq
\end{Proposition}
Observe that the right-hand side is a
rational function in $x$, $y$ and $t$.  We  show in Section~\ref{sec:orbit} that this identity implies
immediately that 19 of the 23 models having a finite group have a
D-finite solution.

The 4 remaining models are Gessel's model $\{x,\bx,xy,\bx\by\}$ (which
we do not solve  in this paper) and the three 
models with steps
$\cS_1=\{\bx,\by,xy\}$, $\cS_2=\{x,y,\bx\by\}$, and
$\cS=\cS_1\cup\cS_2$. For each of these three models, the orbit of $(x,y)$ is 
$$
(x,y)  
 {\overset{\Phi}{\longleftrightarrow}} (\bx \by,y)
 {\overset{\Psi}{\longleftrightarrow}} (\bx \by,x)
 {\overset{\Phi}{\longleftrightarrow}} (y,x)
 {\overset{\Psi}{\longleftrightarrow}} (y,\bx\by)
 {\overset{\Phi}{\longleftrightarrow}} (x,\bx\by)
 {\overset{\Psi}{\longleftrightarrow}} (x,y),
$$
and exhibits  an $x/y$ symmetry.
That is, $(y,x)$ belongs to the orbit of $(x,y)$. Moreover, if
$g((x,y))=(y,x)$, then $\sign (g)=-1$. Thus the right-hand side
of~\eqref{eq:orbit} vanishes, leaving
$$
xyQ(x,y)-\bx Q(\bx\by,y)+\by Q(\bx\by,x)=
 xyQ(y,x)-\bx Q(y,\bx\by)+\by Q(x,\bx\by).
$$
But this identity  directly  follows from the obvious relation $Q(x,y)=Q(y,x)$ and
does not bring much 
information. In Section~\ref{sec:half-orbit}, we  solve these three 
obstinate models by summing
the functional equation over one half of the orbit only. 
Given that $A_{-1}(x)=B_{-1}(x)$, the identity resulting from this
half-orbit summation reads as follows.
\begin{Proposition}[{\bf Half-orbit sums}]\label{prop:half-orbit}
Denote $\cS_1=\{\bx,\by,xy\}$ and  $\cS_2=\{x,y,\bx\by\}$. If the set of
steps is $\cS_1$, $\cS_2$ or $\cS_1\cup \cS_2$,
then
  $$
xyQ(x,y)-\bx Q(\bx\by,y)+\by Q(\bx\by,x)= \frac 
{xy-\bx+\by- 2txA_{-1}(x)Q(x,0)+t\epsilon Q(0,0)} {K(x,y)}.
$$
\end{Proposition}

\noindent{\bf Remark.} For Gessel's walks, the orbit of $(x,y)$ is
shown in Table~\ref{tab:classesD4}. Proposition~\ref{prop:orbit-sum} reads
$$\sum_{g\in G} \sign (g) g(xyQ(x,y))= 0,$$
although no obvious symmetry explains this identity.

\subsection{The roots of the kernel}
Recall that the kernel of the main functional equation (Lemma~\ref{lem:eq-func})
is
$$
K(x,y)=1-t\sum_{(i,j)\in \cS} x^iy ^j.
$$

\begin{Lemma}\label{lem:roots}
  Let 
$$
\Delta(x)= \left(1-tA_0(x)\right)^2-4t^2A_{-1}(x) A_1(x).
$$
Let $Y_0(x)$ and $Y_1(x)$ denote the two roots of the kernel $K(x,y)$,
where  $K(x,y)$ is 
seen as a polynomial in $y$. These roots are
Laurent series in $t$ with coefficients in $\qs(x)$:
 \beq\label{Y-sol}
Y_{0}(x) = \frac{1-tA_0(x)- \sqrt{\Delta(x)}}{2tA_1(x)},
\quad \quad 
Y_{1}(x) = \frac{1-tA_0(x)+ \sqrt{\Delta(x)}}{2tA_1(x)}.
\eeq
Their valuations in $t$ are respectively $1$ and $-1$.
Moreover, $1/K(x,y)$ is a power series in $t$ with coefficients in
$\qs[x,\bx,y,\by]$, and the coefficient of $y^j$ in this series can be
easily extracted using
\beq\label{K-inv}
\frac 1{K(x,y)}= \frac 1{\sqrt{\Delta(x)}} \left( \frac 1{1-\by Y_0(x)}+\frac
1{1-y/ Y_1(x)}-1\right).
\eeq
\end{Lemma}
\begin{proof}
  The equation $K(x,Y)=0$ also reads
\beq\label{Y-alg}
Y=t\left(A_{-1}(x)+YA_0(x)+Y^2A_1(x)\right).
\eeq
Solving this quadratic provides the expressions of  $Y_0(x)$ and
$Y_1(x)$ given above. As $\Delta(x)=1+O(t)$, the series  $Y_1$ has
valuation $-1$ in $t$, and first term $1/(tA_1(x))$.  The equation
$$
Y_0(x)Y_1(x)= \frac {A_{-1}(x)}{A_1(x)}
$$
then shows that $Y_0(x)$ has valuation 1. 
This is also easily seen on \eqref{Y-alg}, which in turn implies that
$Y_0$ has coefficients in $\qs[x,\bx]$. 
The equation
$$
Y_0(x)+Y_1(x)= \frac 1{tA_1(x)}-\frac{A_0(x)}{A_1(x)}
$$
shows that for $n\ge 1$, the coefficient of $t^n$ in $Y_1(x)$ is also
a Laurent polynomial in $x$. This is not true, in general, of the
coefficients of $t^{-1}$ and $t^0$.

The identity~\eqref{K-inv} 
results from a partial fraction expansion in $y$. 
Note that both $Y_0$ and $1/Y_1$ have valuation 1 in $t$. Hence the
expansion in $y$ of $1/K(x,y)$ gives 
\beq\label{1/K-exp}
[y^j] \frac 1{K(x,y)}= 
\left\{
\begin{array}{ll}
 \displaystyle \frac {Y_0(x)^{-j}}{\sqrt{\Delta(x)}} & \hbox{if } j\le
 0\\
\\
 \displaystyle\frac {Y_1(x)^{-j}}{\sqrt{\Delta(x)}} & \hbox{if } j\ge 0.
\end{array}\right.
\eeq
\end{proof}

\subsection{Canonical factorization of the discriminant $\Delta(x)$}
\label{sec:canonical}
%
 The  kernel can be seen as a polynomial in $y$. Its
 discriminant is then a (Laurent) polynomial in $x$:
$$
\Delta(x)= (1-tA_0(x))^2-4t^2A_{-1}(x) A_{1}(x).
$$
Say $\Delta$ has valuation $-\delta$,
and degree $d$ in $x$. 
Then it  admits $\delta+d$ roots $X_i\equiv X_i(t)$, for $1\le
i\le \delta+d$, which are Puiseux series in 
$t$ with complex coefficients. Exactly $\delta$ of them, say $X_1,
\ldots, X_\delta$,   are 
finite (and actually vanish) at 
$t=0$. The remaining $d$ roots, $X_{\delta+1}, \ldots, X_{\delta+d}$,
have a negative valuation in $t$ and thus diverge at $t=0$. (We refer
to~\cite[Chapter~6]{stanley-vol2} for generalities on solutions of algebraic equations with
coefficients in $\cs(t)$.) We  write
$$
\Delta(x) =\Delta_0\Delta_-(\bx) \Delta_+(x),
$$
with  
\begin{eqnarray*}  
\Delta_-(\bx)\equiv \Delta_-(\bx;t)&= &\prod_{i=1}^{\delta} (1-\bx X_i) ,\\
 \Delta_+(x)\equiv \Delta_+(x;t)&=&\prod_{i=\delta+1 }^{\delta+d}
 (1-x/ X_i),
\end{eqnarray*}
and
$$
\Delta_0\equiv \Delta_0(t)= (-1)^\delta
\frac{[\bx^\delta]\Delta(x)}{\prod_{i=1}^\delta X_i}
=
(-1)^d {[x^d]\Delta(x)}{\prod_{i=\delta+1}^{\delta+d} X_i}.
$$
It can be seen  that  $\Delta_0$
(resp. $\Delta_-(\bx)$, $ \Delta_+(x)$) is a formal power
series in $t$ with constant term 1 and coefficients in $\cs$
(resp. $\cs[ \bx ]$,  $\cs[ x ]$). The above factorization is an
instance of the  canonical factorization of series in
$\qs[x,\bx][[t]]$, which was introduced by
Gessel~\cite{gessel-factorization}, and has 
proved useful in several walk problems since
then~\cite{bousquet-slit,bousquet-schaeffer-slit,Bous05}. 
It will play a crucial role in Section~\ref{sec:half-orbit}.

\section{D-finite solutions via orbit sums}\label{sec:orbit}
In this section, we first show that  19 of the 23 models having a
finite group  can be solved from the corresponding orbit sum. This
includes the 16 models having a vertical symmetry, plus 3
others. Then, we work out the latter 3 models in details, obtaining closed
form expressions for the number of walks ending at prescribed positions.
\subsection{A general result}
\begin{Proposition}\label{prop:orbit-sol}
For  the $23$ models associated with a finite group,
except from the four cases $\cS=\{\bx,\by,xy\}$,  $\cS=\{x,y,\bx\by\}$,
$\cS=\{x,y,\bx,\by,xy,\bx\by\}$  and $\cS=\{x,\bx,xy,\bx \by\}$,
the following holds.  The rational function
$$
R(x,y;t)=\frac 1 {K(x,y;t)} \sum_{g\in G}\sign (g) g(xy)
$$
is a power series in $t$ with coefficients in $\qs(x)[y,\by]$ (Laurent
polynomials in $y$, having coefficients in $\qs(x)$). Moreover, the
positive part in $y$ of $R(x,y;t)$, denoted $R^+(x,y;t)$, is a power series in $t$ with
coefficients in $\qs[x,\bx,y]$. Extracting the positive part in $x$ of
$R^+(x,y;t)$  gives $xyQ(x,y;t)$. In brief,
\beq\label{Q-sol}
xyQ(x,y;t)= [x^>] [y^>] R(x,y;t).
\eeq
In particular, $Q(x,y;t)$ is D-finite.
The number of $n$-step walks ending at $(i,j)$ is
\beq\label{qijn-sol}
q(i,j;n)= [x^{i+1} y^{j+1}] \left(\sum_{g\in G}\sign (g) g(xy)\right) S(x,y)^n
\eeq
where 
$$
S(x,y)=\sum_{(p,q) \in \cS}x^p y^q.
$$
\end{Proposition}
\begin{proof}
We begin with the 16 models associated with a group of order 4
(Table~\ref{tab:classesD2}). We will then address the three remaining
cases,  $\cS=\{\bx, y, x\by \}$,
$\cS=\{x,\bx, x\by, \bx y \}$ and $\cS=\{x,\bx, y, \by, x\by, \bx y
\}$.

All models with a group of order 4 exhibit a vertical symmetry. That is,
$K(x,y)=K(\bx,y)$. As discussed in the examples of
Section~\ref{sec:group}, the orbit of $(x,y)$ reads 
$$
(x,y)  
 {\overset{\Phi}{\longleftrightarrow}} (\bx,y)
 {\overset{\Psi}{\longleftrightarrow}} (\bx,C(x)\by)
 {\overset{\Phi}{\longleftrightarrow}} (x,C(x)\by)
 {\overset{\Psi}{\longleftrightarrow}} (x,y),
$$
with $C(x)=\frac{A_{-1}(x)}{A_{1}(x)}$. The orbit sum of
Proposition~\ref{prop:orbit-sum} reads
$$
xyQ(x,y)- \bx y Q(\bx,y)+ \bx\by C(x) Q(\bx, C(x)\by)+ x\by C(x)
Q(x,C(x)\by)
=R(x,y).
$$
Clearly, both sides of this identity are series in $t$ with
coefficients in $\qs(x)[y,\by]$. Let us extract the positive part in
$y$: only the first two terms of the left-hand side contribute, and we
obtain
$$
xyQ(x,y)- \bx y Q(\bx,y)=R^+(x,y).
$$
It is now clear from the left-hand side that $R^+(x,y)$ has
coefficients in $\qs[x,\bx,y]$.
Extracting the positive part in $x$ gives 
the expression~\eqref{Q-sol} for $xyQ(x,y)$,
since the second term  of the left-hand side does not contribute.

Let us now examine the cases $\cS=\{\bx, y, x\by \}$,
$\cS=\{x,\bx, x\by, \bx y \}$ and $\cS=\{x,\bx, y, \by, x\by, \bx y
\}$. For each of them, the orbit of $(x,y)$
consists of pairs of the form $(x^ay^b, x^c y^d)$ for integers $a, b,
c$ and $d$  (see Tables~\ref{tab:classesD3}
and~\ref{tab:classesD4}). This implies that $R(x,y)$ is a series in $t$
with 
coefficients in $\qs[x,\bx,y, \by]$. When extracting the positive part
in $x$ and $y$ from the orbit sum of Proposition~\ref{prop:orbit-sum},
it is easily 
checked, in each of the three cases, that only the term $xyQ(x,y)$
remains in the left-hand side. The  expression of $xyQ(x,y)$ follows.

Proposition~\ref{prop:pos-part} then implies that $Q(x,y;t)$ is
D-finite. The expression of $q(i,j;n)$ follows from a mere coefficient
extraction. 
\end{proof}

\subsection{Two models with algebraic specializations:
  $\{\bx,y,x\by  \}$ and $\{x,\bx,y,\by, x\by , \bx y  \}$} 
\label{sec:tandem}
%
Consider the case  $\cS=\{\bx,y,x\by  \}\equiv\{\WW, \NN, \SE\}$. 
A walk $w$ with steps taken from~$\cS$ remains in the first quadrant
if each of its prefixes contains more \NN\ steps that \SE\ steps, and
more \SE\ steps than \WW\ steps.
These walks are thus in
bijection with Young tableaux of
height at most~3 (Figure~\ref{fig:young}), or, via the Schensted
correspondence~\cite{sagan-book}, with involutions having no
decreasing subsequence of length~4. The enumeration of Young tableaux is
a well-understood topic. In particular, the number of tableaux of a
given shape ---and hence the number of $n$-step walks ending at a prescribed
position--- can be written in closed form using
the hook-length formula~\cite{sagan-book}. It is also
known that the total number of tableaux of size $n$ 
and height at most~3
is the $n$th Motzkin number~\cite{regev}.  
Here, we recover  these two results (and refine the latter),
using  orbit and half-orbit sums.
Then, we show that the case $\cS=\{x,\bx,y,\by,
x\by , \bx y  \}$, which, to our knowledge, has never been solved,
behaves  very similarly. In particular, the total number of $n$-step
walks confined to the quadrant is also related to Motzkin numbers
(Proposition~\ref{prop:double-tandem}). 

\begin{figure}[ht]
\begin{center}
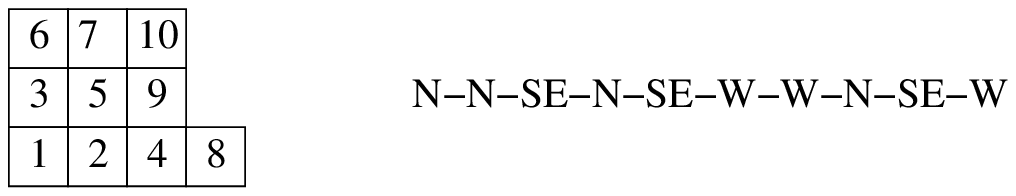
\end{center}
\caption{A Young tableau of height 3 and the corresponding quarter
  plane walk with steps in $\{\WW, \NN, \SE\}$.}
\label{fig:young}
\end{figure}

\begin{Proposition}\label{prop:tandem}
The  \gf \ of walks with steps \WW, \NN,  \SE\ confined to the
  quarter plane  is the non-negative part (in $x$ and $y$) of a
  rational series in $t$ having coefficients in $\qs[x,\bx,y,\by]$:
$$
Q(x,y;t)= [x^\ge y^\ge] \tilde R(x,y;t), \quad \hbox{with} \quad
\tilde R(x,y;t)=\frac{  \left( 1-\bx\by \right)  
\left( 1-\bx^2y \right)  
\left( 1-x\by ^2\right)}{1-t(\bx+y+x\by)}.
$$
In particular, $Q(x,y)$ is D-finite.
 The number of  walks of length $n=3m+2i+j$ ending at $(i,j)$ is
$$
q(i,j;n)=\frac{(i+1)(j+1)(i+j+2) (3m+2i+j)!}{m!(m+i+1)!(m+i+j+2)!}.
$$
In particular,
$$
q(0,0;3m)=\frac{2(3m)!}{m!(m+1)!(m+2)!}\sim \sqrt 3\, \frac{3^{3m}}{\pi m^4},
$$
so that $Q(0,0;t)$ ---and hence $Q(x,y;t)$--- is transcendental.

However, the specialization $Q(x,1/x;t)$ is algebraic of degree $2$:
$$
Q(x,1/x;t)=
{\frac {1-t\bx-
\sqrt {1-2\,\bx t+{t}^{2}\bx^2-4\,{t}^{2}{x}}}{2{x}{t}^{2}}}
.
$$
In particular, the total number of $n$-step walks confined to the quadrant
is the $n^{\hbox{th}}$ Motzkin number:
$$
Q(1,1;t)=
{\frac {1-t-\sqrt {(1+t)(1-3t)}}{2{t}^{2}}}
=\sum_{n\ge 0} t^n \sum_{k=0}^{\lfloor n/2\rfloor} \frac 1{k+1}{n\choose{2k}}{{2k}\choose k}.
$$
\end{Proposition}
\begin{proof}
The orbit of $(x,y)$ under the action of $G$ is shown
in~\eqref{orbit-tandem}. The first result of the proposition is a
direct application of 
Proposition~\ref{prop:orbit-sol}, with $\tilde R(x,y)=R(x,y)/(xy)$. 
It is then an easy task to extract the coefficient of $x^iy^j t^n$ in
$\tilde R(x,y;t)$, using
$$
[x^i y^j] (\bx+y+x\by)^n= \frac{ (3m+2i+j)!}{m!(m+i)!(m+i+j)!}
$$
if $n=3m+2i+j$.

\medskip
The algebraicity of $Q(x, \bx)$ can be proved as follows: let us form
the alternating sum of the three equations obtained from Lemma~\ref{lem:eq-func}
by replacing $(x,y)$ by the first 3 elements of the orbit. These
elements are those in which $y$ occurs with a non-negative exponent.
This gives
$$
 K(x,y)(
xy Q ( x,y ) - \bx y^2Q ( {\bx} {y},y )
+\bx^2 y Q ( { {\bx y}},{\bx} ))
=
 {xy-\bx y^2+\bx^{2}y} - tx^2 Q ( x,0 )-t \bx Q ( 0,{\bx} ).
$$
We now specialize this equation to two values of $y$.
First, replace $y$ by $\bx$: the second and third occurrence of $Q$ in
the left-hand side cancel out, leaving
$$
K(x,\bx)Q ( x,\bx )=
1- tx^2 Q ( x,0 )-t \bx Q ( 0,{\bx} ).
$$
For the second specialization, replace $y$ by the root $Y_0(x)$ of the
kernel (see~\eqref{Y-sol}). 
This is  a well-defined substitution, as $Y_0(x)$ has valuation 1 in
$t$.
The left-hand side vanishes, leaving
$$
0= {xY_0(x)-\bx Y_0(x)^2}+\bx^{2}Y_0(x)
- tx^2 Q ( x,0 )-t \bx Q ( 0,{\bx} ).
$$
By combining the last two equations, we obtain
$$
K(x,\bx)Q ( x,\bx )=
1- xY_0(x)+\bx Y_0(x)^2-\bx^{2}Y_0(x).
$$
The expression of $Q(x, \bx)$ follows.

\end{proof}

The case $\cS=\{\NN,\SS,\WW, \EE,\SE,\NW\}$ is very similar to the previous
one. In particular, the orbit of $(x,y)$  is the same in both
cases. The proof of the previous proposition translates almost
verbatim. Remarkably, Motzkin numbers occur again.
\begin{Proposition}\label{prop:double-tandem}
The 
 \gf \ of walks with steps \NN, \SS, \WW, \EE, \SE, \NW\ confined to the
  quarter plane  is the non-negative part (in $x$ and $y$) of a
  rational function: 
$$
Q(x,y;t)= [x^\ge y^\ge] \tilde R(x,y;t), \quad \hbox{with} \quad
\tilde R(x,y;t)=\frac{(1-\bx\by)\left( 1-\bx^2y \right)  
\left( 1-x\by ^2\right)}{1-t(x+y+\bx+\by+x\by+\bx y)}.
$$
In particular, $Q(x,y;t)$ is D-finite.
The specialization $Q(x,1/x;t)$ is algebraic of degree $2$:
$$
Q(x,1/x;t)=
{\frac {1-tx-t\bx+ \sqrt {(1-t(x+\bx))^2-4t^2(1+x)(1+\bx)}}
{2{t}^{2} \left( 1+x \right)\left( 1+\bx \right)}}.
$$
In particular, the total number of $n$-step walks confined to the
quadrant is $2 ^n$ times the $n^{\hbox{th}}$ Motzkin number:
$$
Q(1,1;t)=
{\frac {1-2t-\sqrt {(1+2t)(1-6t)}}{8{t}^{2}}}
=\sum_{n\ge 0} t^n \sum_{k=0}^{\lfloor n/2\rfloor} \frac {2^n}{k+1}{n\choose{2k}}{{2k}\choose k}.
$$
\end{Proposition}

\subsection{The case $\cS=\{x, \bx, x\by, \bx y\}$}\label{sec:gouyou}
%
Consider quadrant walks made of \EE, \WW, \NW\ and \SE\ steps. 
These walks are easily seen to be in bijection with pairs of
non-intersecting prefixes of Dyck paths: To pass from such a pair to a
quadrant walk, parse the pair of paths from left to right
and assign a direction  (\EE, \WW, \NW\ or \SE) to each pair of steps, as described in Figure~\ref{fig:pair-dyck}. Hence the number of
walks ending at a 
prescribed position is given by a 2-by-2 Gessel-Viennot
determinant~\cite{gessel-viennot} and we can expect a closed form
expression for this number.

\begin{figure}[ht]
\begin{center}
\input{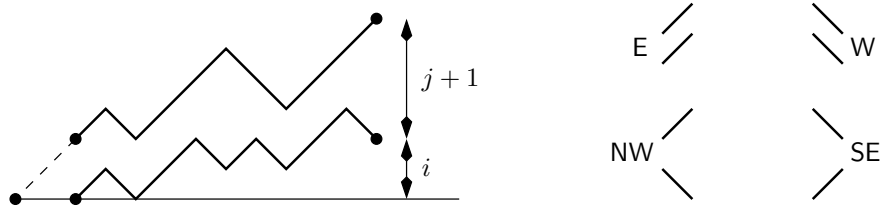}
\end{center}
\caption{A pair of non-intersecting prefixes of Dyck paths
  corresponding to the quarter
  plane walk \EE-\WW-\EE-\EE-\NW-\SE-\WW-\EE-\EE-\NW.}
\label{fig:pair-dyck}
\end{figure}

Also, note that the linear
transformation $(i,j) \mapsto (i+j,j)$ maps these quarter plane walks
bijectively to walks with 
\EE, \WW, \NN\ and \SS\ steps confined to $\{(i,j): 0\le j \le
i\}$. In this form, they
were  studied by Gouyou-Beauchamps, who  proved that the
number of $n$-step walks ending on the $x$-axis is a 
product of Catalan numbers~\cite{gouyou-chemins-montreal}.
His interest in these walks came from a bijection he had established
between walks ending on the $x$-axis and  Young tableaux of height at most
4~\cite{gouyou-tableaux}. 
This bijection is far from being as obvious as the one that relates
tableaux of height at most~3 to quarter plane walks with
\WW, \NN\ and \SW\ steps (Section~\ref{sec:tandem}).

Here, we first specialize Proposition~\ref{prop:orbit-sol} to obtain the number 
of $n$-step walks ending at $(i,j)$ (Proposition~\ref{prop:gouyou}).
Then, we perform an indefinite 
summation on $i$, or $j$, or both $i$ and $j$ to obtain additional closed form
expressions, including  Gouyou-Beauchamps's (Corollary~\ref{coro:gouyou}).

\begin{Proposition}\label{prop:gouyou}
The  \gf \ of walks with steps \EE, \WW, \NW, \SE\ confined to the
  quarter plane  is the non-negative part (in $x$ and $y$) of a
  rational function: 
$$
Q(x,y;t)= [x^\ge y^\ge] \tilde R(x,y;t),
$$
with
$$
\tilde R(x,y;t)=\frac{
  \left(1- \bx \right)  \left(1+ \bx \right) \left( 1-\by \right)
 \left(1-\bx^2 y\right)  \left( 1-x\by \right)  \left( 1+x\by \right)}
{1-t(x+\bx+x\by+\bx y)} .
$$
In particular, $Q(x,y;t)$ is D-finite.
 The number of  walks of length $n=2m+i$ ending at $(i,j)$ is
$$
q(i,j;2m+i)=\frac{(i+1)(j+1)(i+j+2)(i+2j+3)}{(2m+i+1)(2m+i+2)(2m+i+3)^2}
{{2m+i+3}\choose {m-j}}{{2m+i+3}\choose {m+1}}.
$$
In particular,
$$
q(0,0;2m)=
\frac{6 (2m)!(2m+2)!}{m!(m+1)!(m+2)!(m+3)!}
\sim \frac{24\cdot 4^{2m}}{\pi m ^5}
$$
so that $Q(0,0;t)$ ---and hence $Q(x,y;t)$--- is transcendental.
\end{Proposition}

Before we prove this proposition, let us perform summations on $i$ and
$j$.
Recall that a hypergeometric sequence $(f(k))_k$ is \emm Gosper
summable, (in $k$) if there exists
another hypergeometric sequence $(g(k))_k$ such that $f(k)=g(k+1)-g(k)$. In this
case, indefinite summation can be performed in closed form~\cite[Chapter~5]{AB}:
$$
\sum_{k=k_0}^{k_1} f(k)= g(k_1+1)-g(k_0).
$$
 The numbers $q(i,j;n)$ of Proposition~\ref{prop:gouyou} have
 remarkable Gosper properties, from which 
 we now derive Gouyou-Beauchamps's result for walks ending on the
 $x$-axis, and  more.
\begin{Corollary}\label{coro:gouyou}
  The numbers $q(i,j;n)$ are Gosper-summable in $i$ and in $j$. Hence   sums of the form $\sum_{i=i_0}^{i_1} q(i,j;n)$ and
$\sum_{j=j_0}^{j_1} q(i,j;n)$ have  closed form expressions. In
particular, the  number of  walks of length $n$ ending at ordinate $j$ is
$$
q( -, j;n):=\sum_{i\ge 0} q(i,j;n)=
\left\{
\begin{array}{ll}
\displaystyle \frac{(j+1)(2m)!(2m+2)!}{(m-j)!(m+1)!^2(m+j+2)!} 
&\hbox{if } n=2m,\\
\\
\displaystyle \frac{2(j+1)(2m+1)!(2m+2)!}{(m-j)!(m+1)!(m+2)!(m+j+2)!}
 & \hbox{if } n=2m+1.
\end{array}\right.
$$
Similarly,  
the  number of  walks of length $n=2m+i$ ending at abscissa $i$ is
$$
q(i, -\,;2m+i):=\sum_{j\ge 0} q(i,j;2m+i)=
\frac{(i+1)(2m+i)!(2m+i+2)!}{m!(m+1)!(m+i+1)!(m+i+2)!}.
$$
In particular,   as many $2m$-step walks end on the $x$- and
$y$-axes:
$$
q( -, 0;2m)=q(0, -\,;2m)=\frac{(2m)!(2m+2)!}{m!(m+1)!^2(m+2)!}.
$$
\smallskip
The numbers  $q(i, -\,;n)$ and $q( -, j;n)$ defined above are again
Gosper-summable in $i$ and $j$ respectively. Hence  sums of the form 
$$
\sum_{i=i_0}^{i_1}
q(i, -\, ;n)=\sum_{i=i_0}^{i_1}\sum_j q(i,j;n)
\quad 
\hbox{ and  }
\quad 
\sum_{j=j_0}^{j_1}
q(-, j;n)=\sum_{j=j_0}^{j_1}\sum_i q(i,j;n),
$$
 which count walks ending
between certain vertical or horizontal lines, have simple closed form
expressions. In particular,
the  total number of  quarter plane walks of length $n$ is
$$
q(-, - \, ;n):=\sum_{i,j\ge 0} q(i,j;n)=\left\{
\begin{array}{ll}
\displaystyle \frac{(2m)!(2m+1)!}{m!^2(m+1)!^2} &\hbox{if } n=2m,\\
\\
\displaystyle \frac{(2m+1)!(2m+2)!}{m!(m+1)!^2(m+2)!} & \hbox{if } n=2m+1.
\end{array}\right.
$$
The asymptotic behaviours of these numbers are found to be
$$
q(-, - \,; n)\sim c_1 \cdot 4^n/n^2,
\quad 
q(0, - \,; n)\sim c_2 \cdot 4^n/n^3,
\quad
q( -,0; n)\sim c_3 \cdot 4^n/n^3,
$$
 which shows that the series $Q(1,1;t)$, $Q(0,1;t)$ and
$Q(1,0;t)$ are transcendental.
\end{Corollary}
\noindent {\em Proof of Proposition~{\em\ref{prop:gouyou}} and
  Corollary~{\em\ref{coro:gouyou}}.}
The orbit of $(x,y)$ under the action of $G$, of cardinality 8, is shown in
Table~\ref{tab:classesD4}. 
The expression of $Q(x,y;t)$ given in Proposition~\ref{prop:gouyou}
is a direct application of 
  Proposition~\ref{prop:orbit-sol}, with  $\tilde R(x,y)=R(x,y)/(xy)$.
It is then an easy task to extract the coefficient of $x^iy^j t^n$ in
$\tilde R(x,y;t)$, using
$$
 x+\bx+ \bx y+x\by=(1+\by)(x+\bx y) \quad \hbox{and}  \quad
 [x^i y^j] (x+\bx+ \bx y+x\by)^n= {n\choose m+i}{n\choose m-j}
$$
for $n=2m+i$. 
This proves Proposition~\ref{prop:gouyou}.

For the first part of corollary, that is, the expressions of
$q(-,j;n)$ and $q(i, - \, ; n)$, it suffices to check the following identities,
which we have obtained using the implementation of Gosper's algorithm
found in the {\tt sumtools} package of {\sc Maple}: 

\begin{eqnarray*}
  q(2i,j;2m)&=&g_1(i,j;m)-g_1(i+1,j;m),\\
 q(2i+1,j;2m+1)&=&g_2(i,j;m)-g_2(i+1,j;m),\\
 q(i,j;2m+i)&=&g_3(i,j;m)-g_3(i,j+1;m),
\end{eqnarray*}
with
\begin{eqnarray*}
g_1(i,j;m)&=&{\frac {2 \left( 1+j \right)  \left( m+1+2\,i \left( i+1+j \right) 
 \right)  \left( 2\,m \right) !\, \left( 2\,m+1 \right) !}{ \left( m-i
-j \right) !\, \left( m-i+1 \right) !\, \left( m+i+1 \right) !\,
 \left( m+i+j+2 \right) !}},\\
g_2(i,j;m)&=&{\frac {2 \left( 1+j \right)  ( m+1+j \left( 1+2\,i \right) +2
\, \left( 1+i \right) ^{2} )  \left( 2\,m+1 \right) !\, \left( 2
\,m+2 \right) !}{  \left( m-i-j \right) !\, \left( m-i+1 \right) !\left( m+i+2 \right) !\, \left( m+i+j+3 \right) !\,
}},\\
g_3(i,j;m)&=&{\frac { \left( 1+i \right)  \left( m+1+ \left( i+j +1\right)  \left( 
1+j \right)  \right)  \left( 2\,m+i \right) !\, \left( 2\,m+i+2
 \right) !}{ \left( m-j \right) !\, \left( m+1 \right) !\, \left( m+i+
2 \right) !\, \left( m+i+j+2 \right) !}}.
\end{eqnarray*}
These three identities respectively lead to 
$$
q(-,j;2m)=g_1(0,j;m), \quad 
q(-,j;2m+1)=g_2(0,j;m)\quad  \hbox{and} \quad 
q(i,-\, ; 2m+i)=g_3(i,0;m)
$$ 
as stated  in the corollary.

For the second part, we have used the following identities, also
discovered (and proved) using {\sc Maple}:
\begin{eqnarray*}
q(2i,-\,;2m)&=&  g_4(i;m)-g_4(i+1;m),\\
q(2i+1,-\,;2m+1)&=&  g_5(i;m)-g_5(i+1;m),\\
q(-,j;2m) &=&  g_6(j;m)-g_6(j+1;m),\\
q(-,j;2m+1) &=&  g_7(j;m)-g_7(j+1;m),\\
\end{eqnarray*}
with
\begin{eqnarray*}
g_4(i;m)&=&{\frac { \left( 2\,m \right) !\, \left( 2\,m+1 \right) !}{ \left( m-i
 \right) !\, \left( m-i+1 \right) !\, \left( m+i \right) !\, \left( m+
i+1 \right) !}},\\
g_5(i;m)&=& {\frac { \left( 2\,m+1 \right) !\, \left( 2\,m+2 \right) !}{ \left( m-
i \right) !\, \left( m-i+1 \right) !\, \left( m+i+1 \right) !\,
 \left( m+i+2 \right) !}},\\
g_6(j;m)&=&{\frac { \left( 2\,m \right) !\, \left( 2\,m+1 \right) !}{ \left( m-j
 \right) !\,m!\, \left( m+1 \right) !\, \left( m+j+1 \right) !}},\\
g_7(j;m)&=&
 \frac { \left( 2\,m+1 \right) !\, \left( 2\,m+2 \right) !}{ \left( m-j \right) !\, \left( m+1 \right) !\, \left( m+2 \right) !\, \left( m+j
+1 \right) !}.
  \end{eqnarray*}
Note that these identities give  two ways to determine the total
number of $n$-step walks in the quadrant, as
$$
q(-,-\,;2m)= g_4(0;m)=g_6(0;m) \quad \hbox{and} \quad
q(-,-\,;2m+1)= g_5(0;m)=g_7(0;m).
$$
\qed

\section{Algebraic solutions via half-orbit  sums}\label{sec:half-orbit}
In this section we solve in a unified manner the three models whose orbit
has an $x/y$ symmetry: $\cS_1=\{\bx,\by,xy\}$,
$\cS_2=\{x,y,\bx\by\}$ and $\cS=\cS_1\cup \cS_2$. 
Remarkably, in all
three cases the \gf\ $Q(x,y;t)$ is found to be algebraic. 
Our approach uses  the \emm algebraic kernel method, introduced by the
first author to solve the case $\cS=\cS_1$, that is, Kreweras'
model~\cite[Section~2.3]{Bous05}. 
We refer to the introduction for more references on this model.
The case $\cS=\cS_2$ was solved by the second author in~\cite{Mishna-jcta},
and the case $\cS=\cS_1\cup \cS_2$ is, to our knowledge, new.

\medskip
 Recall from Proposition~\ref{prop:half-orbit}
that for each of these three models,
  $$
xyQ(x,y)-\bx Q(\bx\by,y)+\by Q(\bx\by,x)= \frac 
{xy-\bx+\by- 2txA_{-1}(x)Q(x,0)+t\epsilon Q(0,0)} {K(x,y)}.
$$
Extract from this equation the coefficient of $y^0$: in the left-hand
side, only the second term contributes, and its contribution is
$\bx Q_d(\bx)$, where $Q_d(x)\equiv Q_d(x;t)$ is the \gf\ of walks
ending on the diagonal:
$$
Q_d(x;t)=\sum_{n, i\ge 0} t^n x^i q(i,i;n).
$$
The coefficient of $y^0$ in the right-hand side can be easily
extracted using~\eqref{1/K-exp}. This gives
$$
-\bx Q_d(\bx)= \frac 1 {\sqrt{\Delta(x)}} \left(
xY_0(x) -\bx + \frac 1 {Y_1(x)} - 2txA_{-1}(x)Q(x,0)+t\epsilon Q(0,0)
\right),
$$
or, given the expression~\eqref{Y-sol} of $Y_0$ and the fact that $Y_0Y_1=\bx$,
$$
\frac x{tA_1(x)} -\bx Q_d(\bx)= \frac 1 {\sqrt{\Delta(x)}} \left(
\frac{x(1-tA_0(x))}{tA_1(x)} -\bx  - 2txA_{-1}(x)Q(x,0)+t\epsilon Q(0,0)
\right).
$$
Let us write the canonical factorization of $\Delta(x)=\Delta_0
\Delta_+(x) \Delta_-(\bx)$ (see
Section~\ref{sec:canonical}). Multiplying the equation by $A_1(x)
\sqrt{\Delta_-(\bx)}$ gives 
\begin{multline}\label{ht}
\sqrt{\Delta_-(\bx)}\left(
\frac x{t} -\bx A_1(x) Q_d(\bx)\right)=\\
\frac 1 {\sqrt{\Delta_0\Delta_+(x)}} \left(
\frac{x(1-tA_0(x))}{t} -\bx A_1(x)  - 2txA_{-1}(x)A_1(x)Q(x,0)+t\epsilon A_1(x)Q(0,0)
\right).
\end{multline}
Each term in this equation is a 
Laurent series in $t$ with coefficients
in $\qs[x,\bx]$.  Moreover, very few positive powers of $x$ occur in
the left-hand side, while very few negative powers in $x$ occur in the
right-hand side. We will extract from this equation the positive and
negative parts in $x$, and this will give algebraic expressions for the unknown
series $Q_d(x)$ and $Q(x,0)$. 
From now on, we consider each model separately. 

\subsection{The case $\cS=\{\bx,\by,xy\}$}

We have
$A_{-1}(x)=1$, $A_{0}(x)=\bx$, $A_{1}(x)=x$ and $\epsilon =0$. The
discriminant $\Delta(x)$ reads $(1-t\bx)^2-4t^2x$. The curve
$\Delta(x;t)=0$ has a rational parametrization in terms of the series
$W\equiv W(t)$, defined as  the only power series in $t$ satisfying
\beq\label{W-alg}
W=t(2+W^3).
\eeq
Replacing $t$ by $W/(2+W^3)$ in  $\Delta(x)$ gives  the canonical
factorization as
$
\Delta(x)=\Delta_0 \Delta_+(x) \Delta_-(\bx)$
with
$$
\Delta_0=\frac{4t^2}{W^2}, \quad \Delta_+(x)=1-xW^2, \quad
\Delta_-(\bx)= 1-\frac{ W (W^3+4)}{4x}+ \frac{W^2}{4x^2}.
$$
Extracting the positive part in $x$ from~\eqref{ht} gives
$$
{\frac {x}{t}}=
-{\frac { \left( 2\,{t}^{2}{x}^{2}Q \left( x,0
 \right) -x+2\,t \right) W}{2{t}^{2}\sqrt {1-x{W}^{2}}}}+{\frac {W}{t}},
$$
from which we obtain an expression of $Q(x,0)$ in terms of $W$.
Extracting the non-positive part in $x$ from~\eqref{ht} gives
$$
\sqrt {1-{\frac {W \left( {W}^{3}+4 \right) }{4x}}+{\frac {{W
}^{2}}{4{x}^{2}}}}\  \left( {\frac {x}{t}-Q_d \left( \bx \right) } \right) 
-{\frac {x}{t}}=-{\frac {W}{t}},
$$
from which we obtain an expression of $Q_d(\bx)$. We recover the
results of~\cite{Bous05}.
\begin{Proposition}
\label{prop:kreweras}
Let $W\equiv W(t)$ be the power series in $t$ defined by~\eqref{W-alg}.
Then the \gf \ of quarter plane walks formed of \WW, \SS\ and \NE\ steps,
and ending on the $x$-axis is 
$$
Q(x,0;t)= \frac 1 {tx} \left( \frac 1 {2t} - \frac 1 x - 
\left( \frac 1 W -\frac 1 x \right) \sqrt{1-xW^2} \right).$$
Consequently, the length \gf \ of walks ending at $(i,0)$ is
$$[x^i] Q(x,0;t) = \frac{W^{2i+1}}{2.4^i\ t}\left( C_i
-\frac{C_{i+1}W^3}4\right),$$ 
where $C_i={{2i} \choose i}/(i+1)$ is the $i$-th Catalan number.
The Lagrange inversion formula gives the number of such walks of 
length $m=3m+2i$ as
$$
q(i,0;3m+2i)=\frac{4 ^m (2i+1)}{(m+i+1)(2m+2i+1)} {{2i} \choose i}
{{3m+2i} \choose m} .
$$
 The \gf \ of  walks ending on the diagonal is
$$
 Q_d(x;t)=\frac {W-\bx}{\displaystyle t \sqrt{1-xW(1+W^3/4)+x^2W^2/4}}+\bx/t.
$$
\end{Proposition}
Note that $Q(0,0)$ is algebraic of degree 3, $Q(x,0)$ is algebraic of
degree 6, and $Q(x,y)$ (which can be expressed in terms of $Q(x,0)$,
and $Q(0,y)=Q(y,0)$ using the functional equation we started
from) is algebraic of degree 12.

\subsection{The case $\cS=\{x,y,\bx \by\}$}
The steps of this
model are obtained by reversing the steps of the former model. In
particular, the series $Q(0,0)$ counting walks that start and end at
the origin is the same in both models. This observation was used 
in~\cite{Mishna-jcta} to solve the latter case. We present here a
self-contained solution.

We have
$A_{-1}(x)=\bx$, $A_{0}(x)=x$, $A_{1}(x)=1$ and $\epsilon =1$. The
discriminant $\Delta(x)$ is now $(1-tx)^2-4t^2\bx$, and is obtained by
replacing $x$ by $\bx$ in the discriminant of the previous model. In
particular, the canonical factors of $\Delta(x)$ are
$$
\Delta_0=\frac{4t^2}{W^2}, \quad
\Delta_+(x)= 1-\frac{ W (W^3+4)}{4}\, x+ \frac{W^2}{4}\, x^2, \quad
\Delta_-(\bx)=1-\bx W^2,
$$
where $W\equiv W(t)$ is the  power series in $t$ defined by~\eqref{W-alg}.
Extracting the coefficient of $x^0$ in~\eqref{ht} gives
$$
-{\frac {{W}^{2}}{2t}}=-{\frac {W \left( {W}^{4}+4\,W+8t\,Q (0, 0)  \right) }{16t}}
$$
from which we obtain an expression of $Q(0,0)$.
Extracting the non-negative part in $x$ from~\eqref{ht} gives
$$
{\frac {x}{t}}-{\frac {{W}^{2}}{2t}}=
-{\frac { \left(
 2x{t}^{2} Q( x,0 ) -x{t}^{2} Q (0,0) +t-{x}^{2}+{x}^{3}t
 \right) W}
{2x{t}^{2} \sqrt {1-xW({W}^{3}+4)/4+{x}^{2}{W}^{2}/4}}}+{
\frac {W}{2xt}}
$$
from which we obtain an expression of $Q(x,0)$.
Finally, extracting the negative part in $x$ from~\eqref{ht} gives
$$
\left({\frac {x}{t}}- {\frac {Q_d \left(\bx \right) }{x}}
 \right) \sqrt {1-{\frac {{W}^{2}}{x}}}-{\frac {x}{t}}
+{\frac {{W}^{2}}{2t}}=-{\frac  {W}{2xt}},
$$
from which we obtain an expression of $Q_d(\bx)$. We have thus
recovered, and completed, the results of~\cite{Mishna-jcta}. 
Note in particular how simple the number of walks of length $n$ ending
at a diagonal point $(i,i)$ is. 
\begin{Proposition} 
\label{prop:reverse-kreweras}
Let $W\equiv W(t)$ be the power series in $t$ defined by~\eqref{W-alg}.
Then the \gf \ of quarter plane  walks formed of  \NN, \EE\ and \SW\ steps
and ending on the $x$-axis is
$$
Q(x,0;t)= 
{\frac {W \left(4- {W}^{3} \right) }{16t}}
-{\frac {t-{x}^{2}+t{x}^{3}}{2x{t}^{2}}}
-{\frac { \left( 2\,{x}^{2}-x{W}^{2}-W \right) \sqrt {1- xW (W^3+4)/4\,
      +  x^2{W^2}/4}}{2txW}}.
$$
 The \gf \ of  walks ending on the diagonal is
$$
 Q_d(x;t)=\frac{xW(x+W)-2}{2tx^2\sqrt{1-xW^2}}+\frac 1{tx^2}.
$$
Consequently, the length \gf \ of walks ending at $(i,i)$ is
$$[x^i] Q_d(x;t) = \frac{W^{2i+1}}{4^{i+1}t\,(i+2)}
{2i\choose i} \left(2i+4-(2i+1)W^3\right).
$$ 
The Lagrange inversion formula gives the number of such walks of 
length $n=3m+2i$ as
$$q(i,i;3m+2i)=\frac{4 ^m (i+1) ^2}{(m+i+1)(2m+2i+1)} {{2i+1} \choose i}
{{3m+2i} \choose m} .$$
\end{Proposition}
Note that $Q(0,0)$ is algebraic of degree 3, $Q(x,0)$ is algebraic of
degree 6, and $Q(x,y)$ (which can be expressed in terms of $Q(x,0)$,
$Q(0,y)=Q(y,0)$ and $Q(0,0)$ using the functional equation we started
from) is algebraic of degree 12.

\subsection{The case $\cS=\{x,y,\bx,\by,xy,\bx \by\}$}
We have
$A_{-1}(x)=1+\bx$, $A_{0}(x)=x+\bx$, $A_{1}(x)=1+x$ and $\epsilon =1$. The
discriminant $\Delta(x)$ is now $(1-t(x+\bx))^2-4t^2(1+x)(1+\bx)$, and
is symmetric in $x$ and $\bx$. Two of its roots, say $X_1$ and
$X_2$, have valuation 1 in $t$, and the other two roots are $1/X_1$ and
$1/X_2$. By studying the two elementary symmetric functions of $X_1$
and $X_2$ (which are the coefficients of $\Delta_-(\bx)$), we
are led to introduce the power series $Z\equiv Z(t)$, satisfying
\beq\label{Z-alg}
Z=t\, \frac{1-2\,Z+6\,{Z}^{2}-2\,{Z}^{3}+{Z}^{4}}
{(1-Z)^2}
\eeq
and having no constant term.
Replacing $t$ by its expression in terms of $Z$ in $\Delta(x)$
provides  the canonical factors of $\Delta(x)$ as
$$
\Delta_0=\frac{t^2}{Z^2}, \quad
\Delta_+(x)= 1-2\,Z \frac{ 1+{Z}^{2} }{( 1-Z ) ^{2}}\, x+{Z}^{2}{x}^{2},
\quad \Delta_-(\bx)=\Delta_+(\bx).
$$
As in the previous case, extracting from~\eqref{ht} the coefficient of
$x ^0$ gives and expression of $Q(0,0)$:
$$
Q(0,0)={\frac {Z(1-2\,Z  -{Z}^{2}) }{t (1- Z) ^{2}}}.
$$
Extracting then the positive and negative parts of~\eqref{ht} in $x$
gives expressions of $Q(x,0)$ and $Q_d(\bx)$.
\begin{Proposition} 
\label{prop:double-kreweras}
Let $Z\equiv Z(t)$ be the power series with no constant term
satisfying~\eqref{Z-alg}, and denote
$$
\Delta_+(x)= 1-2\,Z \frac{ 1+{Z}^{2} }{( 1-Z ) ^{2}}\, x+{Z}^{2}{x}^{2}.
$$
Then the \gf \ of quarter plane walks formed of \NN, \SS, \EE, \WW,
\SE\ and \NW\ steps, and ending on the $x$-axis is
\begin{multline*}  
Q(x,0;t)= 
\frac
{
\left(Z  ( 1-Z )+2 xZ  -( 1-Z)\, {x}^{2}\right)
 \sqrt {\Delta_+(x)}
}{2tx Z (1- Z )  ( 1+x ) ^{2}}\\
-\frac{Z  (1- Z) ^{2}
+Z  \left( {Z}^{3}+4 {Z}^{2}-5 Z+2 \right) x
- \left(1 -2 Z+7 {Z}^{2}-4 {Z}^{3} \right) {x}^{2}
+{x}^{3}Z  (1- Z )^{2}
}
{2tx Z (1- Z ) ^{2} ( 1+x ) ^{2}}.
\end{multline*}
 The \gf \ of  walks ending on the diagonal is
$$
 Q_d(x;t)=\frac{1-Z -2xZ+x^2Z(Z-1)}{tx(1+x)(Z-1)\sqrt{\Delta_+(x)}}
+\frac{1}{tx(1+x)} 
$$
\end{Proposition}
Note that $Q(0,0)$ is algebraic of degree 4, $Q(x,0)$ is algebraic of
degree 8, and $Q(x,y)$ (which can be expressed in terms of $Q(x,0)$,
$Q(0,y)=Q(y,0)$ and $Q(0,0)$ using the functional equation we started
from) is algebraic of degree 16. However, $Q\equiv Q(1,1)$ has degree 4 only,
and the algebraic equation it satisfies has a remarkable form:
$$
Q \left( 1+tQ \right)  \left(1+2t\,Q+ 2\,{t}^{2}{Q}^{2} \right)=\frac
1{1-6t}.
$$
Also, Motzkin numbers seem to be lurking around, as in
Proposition~\ref{prop:double-tandem}.
\begin{Corollary}
  Let $N\equiv N(t)$ be the only power series in $t$ satisfying
$$
N=t(1+2N+4N^2).
$$
Up to a factor $t$, this series is the \gf\ of the numbers  $2 ^n M_n$, where
$M_n$ is the $n^{\hbox{th}}$ Motzkin number:
$$
N=\sum_{n\ge 0} t^{n+1} \sum_{k=0}^{\lfloor n/2\rfloor} \frac {2^n}{k+1}{n\choose{2k}}{{2k}\choose k}.
$$
Then the \gf\ of all walks in the quadrant with steps N, S, E, W, SE
and NW is
$$
Q(1,1;t)= \frac 1{2t} \left( \sqrt{\frac{1+2N}{1-2N}}-1\right),
$$
and the \gf\ of walks in the quadrant ending at the origin is
$$
Q(0,0;t)= \frac{(1+4N)^{3/2}}{2Nt} -\frac 1 {2t^2}-\frac 2 t.
$$
\end{Corollary}
The proof is elementary once the algebraic equations satisfied by 
$Q(1,1)$ and $Q(0,0)$ are obtained.

\section{Final comments and questions}\label{sec:questions}

The above results raise, in our opinion, numerous natural
questions. Here are some of them.
The first two families of questions are of a purely combinatorial, or
even bijective, nature (``explain why some results are so simple''). 
Others are more closely related to the method used in this paper.
We also raise a question of an algorithmic nature.
\subsection {Explain closed form  expressions}
We have  obtained remarkable hypergeometric   expressions for the number of
  walks in many cases (Propositions~\ref{prop:tandem}
  to~\ref{prop:reverse-kreweras}). Are there direct combinatorial
  explanations?
Let us underline a few examples that we consider worth investigating.
\begin{enumerate}
\item[---] {\bf Kreweras' walks and their reverse:}
The number of Kreweras' walks ending at $(i,0)$ is remarkably simple
(Proposition~\ref{prop:kreweras}). A 
combinatorial explanation has been found when $i=0$, in connection
with the enumeration of planar   triangulations~\cite{Bern07}. To our
knowledge, the generic case remains open. If we consider instead the
reverse collection of steps (Proposition~\ref{prop:reverse-kreweras}),
then it is the number of walks ending at $(i,i)$ that is remarkably
simple. This is a new result, which we would like to see explained in
a more combinatorial manner.
\item[---] {\bf Motzkin numbers:} this famous sequence of numbers
  arises in the solution of the cases $\cS=\{\bx, y, x\by\}$ and
  $\cS=\{x,\bx, y,\by,  x\by , \bx y\}$ (Propositions~\ref{prop:tandem}
    and~\ref{prop:double-tandem}). The first problem is 
    equivalent to the enumeration of 
 involutions with no decreasing subsequence of length 4, and
the occurrence of Motzkin numbers follows from restricting a bijection
of Fran\c con and Viennot~\cite{francon-viennot}. 
The solution to the second problem is, to our
    knowledge, new, and deserves a more combinatorial
    solution. 
Can one find a direct explanation for why the
respective counting sequences for the total number of walks of these two
models differ by a power of 2?
Is there a connection with the $2^n$   phenomenon of~\cite{duchi-sulanke}?
\item[---] {\bf Gessel's walks:} although we have not solved this
  case ($\cS=\{x, \bx, xy,\bx\by\}$) in this paper, we cannot resist
  advertising Gessel's former conjecture, which has now become 
  Kauers--Koutschan--Zeilberger's theorem:
$$
q(0,0;2n)=16^n \frac{  (5/6)_n (1/2)_n}{ (5/3)_n (2)_n}.
$$
%
\end{enumerate}
Certain other closed form expressions obtained in this paper are less
mysterious. As discussed at the beginning of Section~\ref{sec:gouyou},
walks with \EE, \WW, \NW\ and \SE\ steps are in bijection with pairs
of non-intersecting walks. The Gessel-Viennot method (which, as our
approach, is  an inclusion-exclusion argument) expresses the
number of walks ending at $(i,j)$ as a 2-by-2
determinant, thus justifying the closed form expressions of
Proposition~\ref{prop:gouyou}. The extension of this theory by
Stembridge~\cite{stembridge-pfaffian} allows one to let $i$, or $j$,
or both $i$ and $j$ vary, and the number of walks is now expressed as
a pfaffian. Hence the closed forms of Corollary~\ref{coro:gouyou} are
not unexpected. 
However, one may try to find direct proofs not involving the
inclusion-exclusion principle. 
Moreover, the following question may be interesting \emm per se,:
\begin{enumerate}
 \item[---]{\bf Walks with \EE, \WW, \NW\ and \SE\ steps:} can one
   explain bijectively why as many $2m$-step walks end on
   the $x$- and $y$-axes? Recall that those ending on the $x$-axis
   were counted bijectively in~\cite{gouyou-chemins-montreal}.
\end{enumerate}

 Another well-understood case 
is that of quarter plane walks with \NN, \EE,
  \SS\ and \WW\ steps. The number of such walks ending at the origin
  is a product of Catalan numbers, and this has been explained
  bijectively, first in a recursive manner~\cite{cori-dulucq-viennot},
  and more recently directly, using again certain planar maps as intermediate
  objects~\cite{bernardi-tree-rooted}. Another argument, based on the
  reflection principle and thus involving minus signs, appears
  in~\cite{guy-bijections} and applies to
  more general endpoints.

\subsection {Explain   algebraic  series}
A related  problem is to explain combinatorially, via a direct
construction, why the three models of
  Section~\ref{sec:half-orbit} have algebraic \gfs. Given the
  connection between Kreweras' walks and planar
  triangulations~\cite{Bern07}, this   could be of the same complexity
  as proving directly that families of 
  planar maps have an algebraic \gf. (Much progress has been made
  recently on this problem by Schaeffer, Di Francesco and their
  co-authors.)
And what about   Gessel's walks 
  (with steps \EE, \WW, \NE\ and \SW), which we have not solved in this
  paper, but have very recently been proved
  to have an algebraic \gf\ as well~\cite{BoKa08}?

\subsection {Models with a vertical symmetry}
When $\cS$ is invariant by a reflection across a vertical axis, the
group $G(\cS)$ has cardinality 4 and Proposition~\ref{prop:orbit-sol}
gives the \gf\ $Q(x,y;t)$ as the positive part of a rational
function. We have not worked out the coefficient extraction in any of
these 16 cases. This may be worth doing, with the hope of finding
 closed form expressions in some cases. However, according
 to~\cite{bostan-kauers},  there is little hope
 to find an algebraic solution for $Q(1,1;t)$.

\subsection {Models with an infinite group}  
Two of the 56 models that are associated with an infinite group
(Table~\ref{tab:infinite}) have been proved to have a non-D-finite
  \gf~\cite{Mishna-Rechni}. 
We conjecture that this holds for all models with an
  infinite group.  This conjecture is based on our experimental
  attempts to discover a differential equation satisfied by the \gf,
  and much strengthened by the further attempts
  of Bostan and Kauers~\cite{bostan-kauers},
which are based on the calculation of 1000 terms of each \gf. 
It also relies on the fact that all
  equations with two catalytic variables and an infinite group that
  have been solved so far have a non-D-finite solution.
How could one prove this conjecture, for instance in the  case
$\cS=\{\NN, \EE, \NE, \SW\}$? 
And what is the importance, if any, that can be attributed to the fact
that the step sets with finite groups either exhibit a symmetry across
the $y$-axis or have a vector sum of zero? 

\subsection{Automatic derivation of differential equations}
Our D-finite but transcendental solutions are expressed as the
positive part (in $x$ and $y$) of a 3-variable rational series in $t$,
$x$ and $y$ (Proposition~\ref{prop:orbit-sol}). Can one derive automatically from these expressions differential
equations satisfied by $Q(0,0;t)$ and $Q(1,1;t)$? This would be a
convenient way to fill in the gap between our work and the
paper~\cite{bostan-kauers}, where differential equations are conjectured for
the series $Q(1,1;t)$.

 For the algebraic solutions of Section~\ref{sec:half-orbit},
it is easy  to derive from our results first an algebraic equation
satisfied by  $Q(0,0;t)$ (or $Q(1,1;t)$), and then a differential
equation satisfied by this series, using the {\sc{Maple}} package
{\sc{Gfun}}~\cite{gfun}.

\subsection {Variations and extensions}
%
%
It is natural to ask to which similar problems the approach used
in this paper could be adapted. To make this question more precise,  let us
underline that such problems may involve, for instance, putting
weights on the walks, allowing more general steps, or considering
higher dimensions. However, the very first question is whether Gessel's
model can be dealt with using the ideas of this paper!

\medskip\noindent 
{\bf Weighted paths.} 
One may try to adapt our approach to solve refined enumeration
problems. For instance, some authors have studied the
  enumeration of walks in a wedge,  keeping track not only of the
  number of steps, but also of the number of \emm contacts,, or \emm visits,
  to the boundary
  lines~\cite{buks,Nied05,niederhausen05-bis}. Of course, other
  statistics could be considered.

Another natural way to add weights, of a more probabilistic nature, consists in
studying Markov chains confined to the quarter plane. The weight of a
walk is then its probability. An entire book is devoted
to the determination of the \emm stationary
distributions, of such chains~\cite{fayolle-livre}. These
distributions are governed 
by functional equations similar to ours, but \emm without the length
variable $t$, (since only the stationary regime is considered). This
difference makes the problem rather different in nature, and indeed, the tools
involved in~\cite{fayolle-livre} are much
more analytic than algebraic. A natural way to set the problem back in
the algebraic playground (to which our power series methods belong) is to
keep track of the length of the 
trajectories, which boils down  to studying  the law of the chain at
time $n$. This was done in~\cite{Bous05} for  a probabilistic version
of Kreweras' walks, using a variant of the
method presented in this paper. An asymptotic analysis of the
solution should then yield the limiting/stationary
distribution. This was 
however \emm not, done in~\cite{Bous05}. Instead, we enriched our algebraic
approach  with a few basic analytic arguments  to solve directly the
equation that describes the stationary distribution. This solution is
in our opinion more elementary  than the original 
ones~\cite{fayolle-livre,flatto-hahn}. It is worth noting that the bivariate
series that describes the stationary distribution is algebraic, and that this Kreweras chain is actually
the main algebraic example of~\cite{fayolle-livre}. In view in the
results  presented in this paper, it is natural to
ask if one could  design probabilistic versions of the 
other three algebraic models (Propositions~\ref{prop:reverse-kreweras}
and~\ref{prop:double-kreweras}, plus Gessel's model)  that would also
yield algebraic stationary distributions.

Conversely, it is natural to ask whether certain tools
from~\cite{fayolle-livre}, other than the group of the walk, could be
adapted to our power series context. 
We are thinking in particular of
the material of Chapter 4, which is devoted to the case where the
group of the walk is finite, and (under this hypothesis) to the
conditions under which the 
stationary distribution has an algebraic \gf.

\medskip\noindent 
{\bf More general steps.} The fact that we only allow ``small'' steps
plays a crucial role in our approach, and more precisely in the
definition of the group of the walk (Section~\ref{sec:group}). However, this
does not mean that models with larger steps are definitely beyond
reach. First, it is always possible to write a functional equation
for the series $Q(x,y;t)$, based on a step-by-step construction of the
walk. If no step has a coordinate smaller than $-1$, the right-hand
side of the equation only involves $Q(0,y;t)$ and $Q(x,0;t)$, but
otherwise more unknown functions, depending exclusively on $x$ or
$y$, appear. Another important difference with the present setting is
that the kernel has now degree larger than 2 (in $x$ or $y$). One can
still define a group, but acting on pairs $(x,y)$ than cancel the
kernel. We refer to~\cite{BoPe03} for the solution of a simple example, with
steps $(2,-1)$ and $(-1,2)$.

\medskip\noindent 
{\bf Higher dimension.}
Finally, a natural question is to address 3-dimensional problems as
those studied experimentally in~\cite{bostan-kauers}. Provided one
focusses on walks with small steps, the key
ingredients of our approach --- the functional equation and the group
of the walk --- can indeed be adapted in a straighforward manner to
this higher-dimensional context.

\section{Tables}\label{sec:tables}
The tables below list the 79 step sets $\cS$ we consider, classified
according to the cardinality of the group $G(\cS)$. The first three tables
contain sets for which $G(\cS)$ has cardinality 4, 6 and 8
respectively. The orbit of $(x,y)$ under the action of this group is
listed in the second 
column. The third column lists the steps of $\cS$. The fourth one displays
the numbers $q(1,1;n)$ and $q(0,0;n)$ that respectively count all
quarter plane walks and quarter plane walks ending at the origin. We
have given the reference of these sequences when they appear in the
 Encyclopedia of Integer Sequences~\cite{sloane}. In
the rightmost column, we give references on this model, both in this
paper and in other papers.
%


\newpage
\begin{longtable}[t]{|c|c|c|p{6cm}|l|}\hline
$\#$&$G(\cS)$ & ${\cS}$ &\begin{minipage}{4cm}\mbox{}\\$q(1,1;n)$\\$q(0,0;n)$\\\end{minipage}& References\\ \hline
\TE{1}{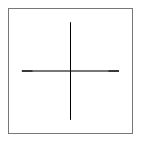}%
{\multirow{4}{*}{%
\begin{minipage}{2cm}
\vfill 
$(x,y)$,$(\bar x, y)$,\\$(\bar x, \bar y)$,$(x, \bar y)$
\vfill
\end{minipage}}}
{1, 2, 6, 18, 60, 200, 700, 2450, 8820, 31752 (A005566)}
{1, 0, 2, 0, 10, 0, 70, 0, 588, 0, 5544, 0, 56628 (A005568)}%
{Transcendental D-finite}
{Prop.~\ref{prop:orbit-sol}, \cite{bousquet-versailles,cori-dulucq-viennot,bernardi-tree-rooted,gessel-zeilberger,guy-bijections}}%
\TE{2}{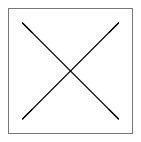}{}%
{1, 1, 4, 9, 36, 100, 400, 1225, 4900, 15876 (A018224)}
{1, 0, 1, 0, 4, 0, 25, 0, 196, 0, 1764, 0, 17424 (A001246)}%
{Transcendental D-finite}%
{Prop.~\ref{prop:orbit-sol}, \cite{bousquet-versailles,gessel-zeilberger,poulalhon-schaeffer}}%
\TE{3}{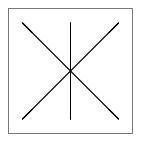}{}%
{1, 2, 10, 39, 210, 960, 5340, 26250, 148610, 761796}
{1, 0, 2, 0, 18, 0, 255, 0, 4522, 0, 91896, 0, 2047452}
{D-finite}
{Prop.~\ref{prop:orbit-sol}, \cite{bousquet-versailles,gessel-zeilberger}}
\TE{4}{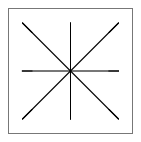}{}%
{1, 3, 18, 105, 684, 4550, 31340, 219555, 1564080}
{1, 0, 3, 6, 38, 160, 905, 4830, 28308, 166992}
{D-finite}
{Prop.~\ref{prop:orbit-sol},  \cite{bousquet-versailles,gessel-zeilberger}}
\hline
\TsE{5}{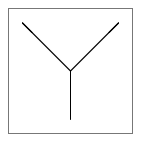}{
$(x,y),(\bar x,y)$,\\ $(\bar x, \bar y \frac{1}{x+\bar x})$,\\ $(x,  \bar y \frac{1}{x+\bar x}) $}
{1, 1, 3, 7, 19, 49, 139, 379, 1079, 3011}
{1, 0, 0, 0, 2, 0, 0, 0, 28, 0, 0, 0, 660, 0, 0, 0, 20020}
{Transcendental D-finite}
{Prop.~\ref{prop:orbit-sol}, \cite{bousquet-versailles,Mishna-jcta}}
\TE{6}{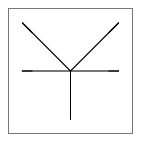}{}
{1, 2, 9, 34, 151, 659, 2999, 13714, 63799, 298397, 1408415, 6678827 }
{1, 0, 1, 3, 4, 20, 65, 175, 742, 2604, 9072, 36960,139392 }
{D-finite}
{Prop.~\ref{prop:orbit-sol}, \cite{bousquet-versailles}}
\hline
\TsE{7}{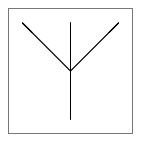}{
$(x,y),(\bar x,y)$,\\
$(\bar x,\bar y\,\frac {1}{x+1+\bar x})$,\\
$(x,\bar y\,\frac {1}{x+1+\bar x}) $}
{1, 2, 7, 23, 84, 301, 1127, 4186, 15891, 60128, 230334}
{1, 0, 1, 0, 4, 0, 20, 0, 126, 0, 882, 0, 6732}
{D-finite}
{Prop.~\ref{prop:orbit-sol}, \cite{bousquet-versailles}}
\TE{8}{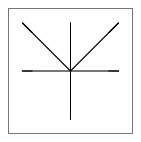}{}
{1, 3, 15, 74, 392, 2116, 11652, 64967, 365759, 2074574}
{1, 0, 2, 3, 12, 40, 145, 560, 2240, 9156, 38724, 166320, 728508}
{D-finite}
{Prop.~\ref{prop:orbit-sol}, \cite{bousquet-versailles}}
\hline
\TsE{9}{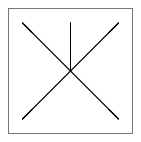}{
$(x,y),(\bar x, y)$,\\  
$(\bar x, \bar y\, \frac {x+\bar x}{x+1+\bar x})$,\\ 
$(x, \bar y\, \frac {x+\bar x}{x+1+\bar x})$}
{1, 2, 8, 29, 129, 535, 2467, 10844, 50982, 231404}
{1, 0, 1, 0, 6, 0, 55, 0, 644, 0, 8694, 0, 128964}
{D-finite} 
{Prop.~\ref{prop:orbit-sol}, \cite{bousquet-versailles}}
\TE{10}{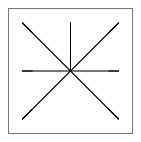}{}
{1, 3, 16, 86, 509, 3065, 19088, 120401, 771758}
{1, 0, 2, 3, 20, 60, 345, 1400, 7770, 36876, 204876}
{D-finite}
{Prop.~\ref{prop:orbit-sol}, \cite{bousquet-versailles}}
\hline
\TsE{11}{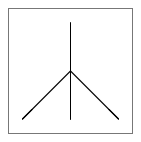}{
$(x,y),(\bar x,y)$,\\
$(\bar x,\bar y\,(x+1+\bar x))$,\\
$(x,\bar y\,(x+1+\bar x))$
}
{1, 1, 3, 5, 17, 34, 121, 265, 969, 2246, 8351, 20118}
{1, 0, 1, 0, 4, 0, 20, 0, 126, 0, 882, 0, 6732}
{D-finite}
{Prop.~\ref{prop:orbit-sol}, \cite{bousquet-versailles}}
\TE{12}{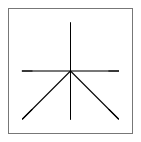}{}
{1, 2, 7, 26, 105, 444, 1944, 8728, 39999, 186266}
{1, 0, 2, 3, 12, 40, 145, 560, 2240, 9156, 38724, 166320, 728508}
{D-finite}
{Prop.~\ref{prop:orbit-sol}, \cite{bousquet-versailles}}
\hline
\TsE{13}{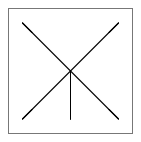}{
$(x,y), (\bar x,y)$,\\ 
$(\bar x,\bar y\, \frac{x+1+\bar x}{x+\bar x})$,\\
$(x,\bar y\, \frac{x+1+\bar x}{x+\bar x})$}%
{1, 1, 5, 13, 61, 199, 939, 3389, 16129, 61601, 295373}
{1, 0, 1, 0, 6, 0, 55, 0, 644, 0, 8694, 0, 128964}%
{D-finite}
{Prop.~\ref{prop:orbit-sol}, \cite{bousquet-versailles}}
\TE{14}{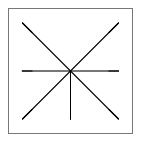}{}
{1, 2, 11, 49, 277, 1479, 8679, 49974, 301169, 1805861}
{1, 0, 2, 3, 20, 60, 345, 1400, 7770, 36876, 204876}
{D-finite}
{Prop.~\ref{prop:orbit-sol}, \cite{bousquet-versailles}}
\hline
\TsE{15}{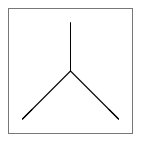}{
$(x,y),(\bar{x},y)$,\\
$(\bar x,\bar{y}(x+\bar{x}))$,\\
$(x,\bar{y}(x+\bar{x}))$}
{1, 1, 2, 3, 8, 15, 39, 77, 216, 459, 1265, 2739, 7842}
{1, 0, 0, 0, 2, 0, 0, 0, 28, 0, 0, 0, 660}
{Transcendental D-finite}
{Prop.~\ref{prop:orbit-sol}, \cite{bousquet-versailles,Mishna-jcta}}
\TE{16}{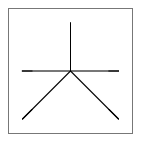}{}
{1, 2, 6, 21, 76, 290, 1148, 4627, 19038, 79554, 336112}
{1, 0, 1, 3, 4, 20, 65, 175, 742, 2604, 9072, 36960}
{D-finite}
{Prop.~\ref{prop:orbit-sol}, \cite{bousquet-versailles}}
\hline
\caption{The group $G(\cS)$ is isomorphic to $D_2$.  
These 16 models have a D-finite generating function.} 
\label{tab:classesD2}
\end{longtable}
%

\begin{longtable}[t]{|c|c|c|p{6cm}|l|}\hline
$\#$&$G(\cS)$ & ${\cS}$ &\begin{minipage}{4cm}\mbox{}\\$q(1,1;n)$\\$q(0,0;n)\\$\end{minipage}& References\\ \hline
\TtE{1}{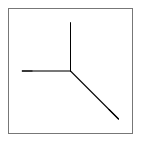}{ 
$(x,y), (\bar x y,y)$,\\  
$(\bar x y, \bar x), (\bar y, \bar x)$,\\
$(\bar y, \bar y x), (x,\bar y x)$
}
{1, 1, 2, 4, 9, 21, 51, 127, 323, 835, 2188, 5798 (A001006)}
{1, 0, 0, 1, 0, 0, 5, 0, 0, 42, 0, 0, 462 (A005789)}
{Algebraic}
{Prop.~\ref{prop:tandem}, \cite{gessel-zeilberger}}

\TE{2}{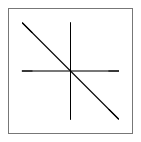}{} 
{1, 2, 8, 32, 144, 672, 3264, 16256, 82688 (A129400)}
{1, 0, 2, 2, 12, 30, 130, 462, 1946, 7980, 34776, 153120}
{Algebraic} 
{Prop.~\ref{prop:double-tandem}, \cite{gessel-zeilberger}}
\hline
\TtE{3}{01001010}{
$(x,y), (\bar x \bar y,y)$,\\ 
$(\bar x \bar y, x)$, $(y,x)$,\\
$(y,  \bar x \bar y),   (x,\bar x\bar y)$}
{1, 1, 3, 7, 17, 47, 125, 333, 939, 2597, 7183}
{1, 0, 0, 2, 0, 0, 16, 0, 0, 192, 0, 0, 2816 (A006335)}
{Algebraic}
{Prop.~\ref{prop:kreweras}, \cite{kreweras, Bous05}}
\TE{4}{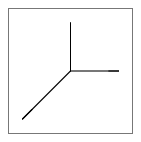}{} 
{1, 2, 4, 10, 26, 66, 178, 488, 1320, 3674, 10318}
{1, 0, 0, 2, 0, 0, 16, 0, 0, 192, 0, 0, 2816 (A006335)}
{Algebraic}
{Prop.~\ref{prop:reverse-kreweras}, \cite{Mishna-jcta}}

\TE{5}{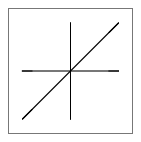}{} 
{1, 3, 14, 67, 342, 1790, 9580, 52035, 285990}
{1, 0, 3, 4, 26, 80, 387, 1596, 7518, 34656, 167310}
{Algebraic}
{Prop.~\ref{prop:double-kreweras}}
\hline 
\caption{The group $G(\cS)$ is isomorphic to $D_3$.  These five models have a D-finite generating function, and the last three  even have an algebraic \gf.}
\label{tab:classesD3}
\end{longtable}

\newpage
\begin{longtable}[t]{|c|c|c|p{6cm}|l|}\hline
$\#$&$G(\cS)$ & ${\cS}$ &\begin{minipage}{4cm}\mbox{}\\$q(1,1;n)$\\$q(0,0;n)\\$\end{minipage}& References\\ \hline
\TE{1}{00110011}  
{
\begin{minipage}{3cm}
\mbox{}\\$(x, y), (y\bar x, y),$\\ $(y\bar x, y{\bar x}^2), (\bar x, y{\bar x}^2)$,\\
$ (\bar x, \bar y), (x\bar y, \bar y),$\\$(x\bar y, x^2\bar y), (x,\bar y x^2)$\\
\end{minipage}}
{1, 1, 3, 6, 20, 50, 175, 490, 1764, 5292 (A005558)}
{1, 0, 1, 0, 3, 0, 14, 0, 84, 0, 594, 0, 4719 (A005700)}
{Tran. D-finite}
{Prop.~\ref{prop:gouyou}, \cite{gessel-zeilberger,gouyou-chemins-montreal,Nied05}}
\hline
\TE{2}{01100110} 
{\begin{minipage}{3cm}
\mbox{}\\$(x, y), (\bar x \bar y, y),$\\ $(\bar x \bar y, x^2y), (\bar x, x^2y)$,\\
$ (\bar x, \bar y), (xy, \bar y),$\\ $(xy, {\bar x}^2 \bar y), (x, \bar y {\bar x}^2)$\\
\end{minipage}
}
{1, 2, 7, 21, 78, 260, 988, 3458, 13300, 47880 (A060900)}
{1, 0, 2, 0, 11, 0, 85, 0, 782, 0, 8004, 0, 88044 (A135404)}
{Algebraic}
{\cite{KaKoZe08, BoKa08,petkovsek-wilf}}
\hline 

\caption{The group $G(\cS)$ is isomorphic to $D_4$. Both models have a D-finite generating function, and the second one even has an algebraic \gf.}
\label{tab:classesD4}
\end{longtable}

\begin{table}[h] \center
\begin{tabular}[t]{cccccccccc}\hline
\\ 
\tpic{10010001}&\tpic{01010001}& \cite{Mishna-Rechni}\\\\\hline 
\\ 
\tpic{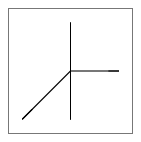}&\tpic{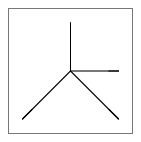}&\tpic{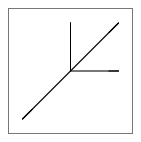}&\tpic{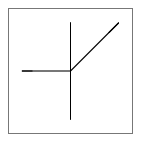}&\tpic{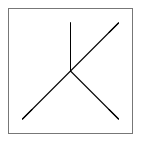}&\tpic{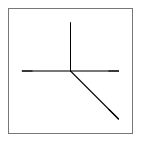}&\tpic{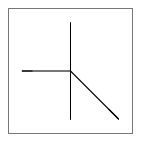}&\tpic{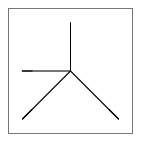}&\tpic{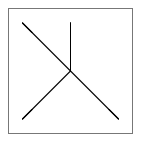}\\ 
\tpic{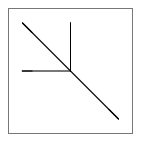}&\tpic{10110001}&\tpic{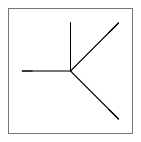}&\tpic{11010001}&\tpic{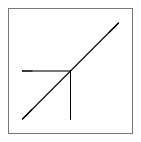}&\tpic{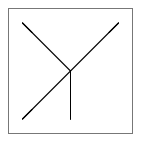}&\tpic{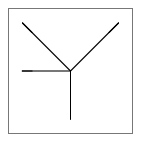}&\tpic{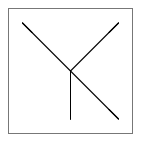}\\\\\hline 
\\ 
\tpic{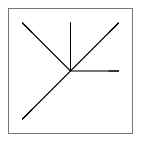}&\tpic{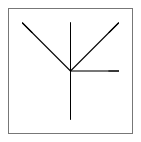}&\tpic{11110001}&\tpic{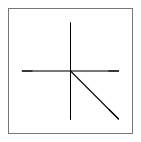}&\tpic{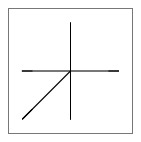}&\tpic{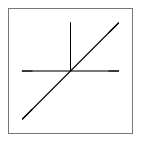}&\tpic{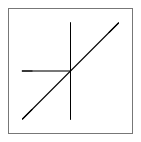}&\tpic{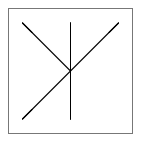}&\tpic{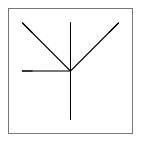}\\ 
\tpic{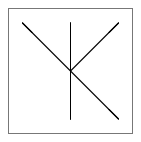}&\tpic{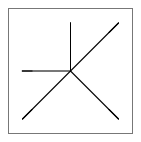}&\tpic{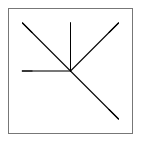}&\tpic{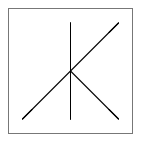}&\tpic{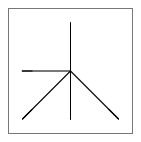}&\tpic{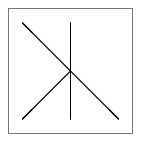}&\tpic{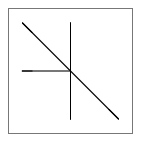}&\tpic{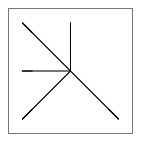}&\tpic{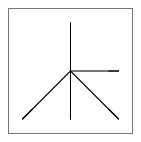}\\ 
\tpic{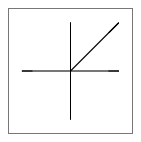}&\tpic{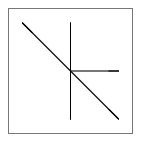}&\tpic{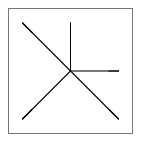}&\tpic{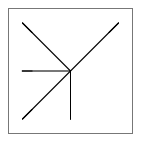}&\tpic{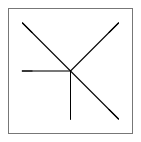}\\\\\hline 
\\ 
\tpic{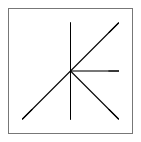}&\tpic{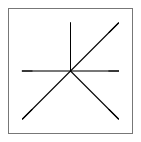}&\tpic{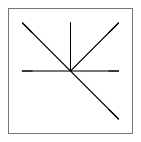}&\tpic{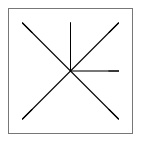}&\tpic{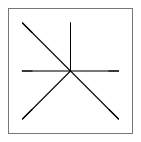}&\tpic{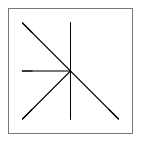}&\tpic{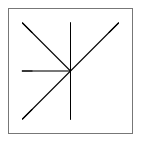}&\tpic{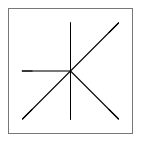}&\tpic{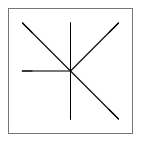}\\ 
\tpic{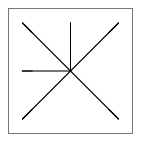}&\tpic{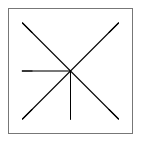}\\\\\hline 
\\ 
\tpic{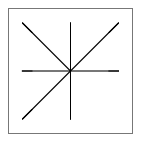}&\tpic{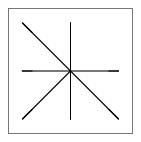}&\tpic{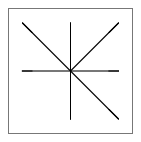}\\\\\hline 
\\
\end{tabular}

\caption{These 56 step sets are associated with an infinite group.}
\label{tab:infinite}
\end{table}

\begin{longtable}{clll}
 & Fixed point  & Condition & $\cchi(X)$
\\ \hline
\\
Four steps\\
\hline
\tabB{1}{10101100}{\left(a,{a}^{-2}\right)}{{a}^{4}+{a}^{3}-1}{ X^8+9X^7+31X^6+62X^5+77X^4+62X^3+31X^2+9X+1}
\tabB{2}{10110100}
%
{\big(a,{
\frac {1-{a}^{2}}{{a}^{2}}}\big)}
{ {a}^{5}-{a}^{4}+{a}^{3}+2\,{a}^{2}-1 }
%
%
{{X}^{10}+18\,{X}^{9}+125\,{X}^{8}+439\,{X}^{7}+897\,{X}^{6}+1131\,{X}^
{5}+897\,{X}^{4}+439\,{X}^{3}+125\,{X}^{2}+18\,X+1}
\tabB{3}{11100100}{\left(a,a\right)}{a^4+a^3-1}{X^8-19X^7-X^6-124X^5+3X^4-124X^3-X^2-19X+1}
\tabB{4}{11010100}{\left(b^2/2,b\right)}{ 
b^6+b^4-4 }{{X}^{6}+9\,{X}^{5}+59\,{X}^{4}+70\,{X}^{3}+59\,{X}^{2}+9\,X+1}
\\ \\
Five steps\\
\hline
\tabB{6}{11100101}{\left(-1,1\right)}{}{2\,{X}^{2}+3X+2}
\tabB{7}{11101001}{
\left(a, \frac{a^2}{1-a^2}\right)}
%
{a^5+a^3+2a^2-1}
%
%
%
{{X}^{10}+16\,{X}^{9}+106\,{X}^{8}+371\,{X}^{7}+764\,{X}^{6}+967\,{X}^{
5}+764\,{X}^{4}+371\,{X}^{3}+106\,{X}^{2}+16\,X+1}
 \tabB{8}{10101110}{\left(a,a\right)}
%
{a^3-a-1}{{X}^{6}+8\,{X}^{5}+28\,{X}^{4}+41\,{X}^{3}+28\,{X}^{2}+8\,X+1
}
\tabB{9}{11100110}{\left(a,a^{-2}\right)}{a^3-a-1}{{X}^{6}+2\,{X}^{5}+6\,{X}^{4}+5\,{X}^{3}+6\,{X}^{2}+2\,X+1}
\tabB{10}{11001101}{
\left(\frac{2b^2}{1-b^2},b\right)}{3b^6+b^4+b^2-1}{
 {X}^{6}+14\,{X}^{5}+87\,{X}^{4}+100\,{X}^{3}+87\,{X}^{2}+14\,X+1
}
\tabB{11}{11010110}{\left(a,1/a\right)}{a^2+a+1}
{{X}^{4}+17\,{X}^{3}+81\,{X}^{2}+17\,X+1}
%
\\ \\
Six steps\\
\hline
\tabB{13}{11111100}{\left(a,2+a-3a^2-3a^3\right)}{3a^6+6a^5+2a^4-5a^3-4a^2+1}
%
{{X}^{12}+23\,{X}^{11}+283\,{X}^{10}+1861\,{X}^{9}+7461\,{X}^{8}+14225\,{X}^{7}
+18249\,{X}^{6}+14225\,{X}^{5}+7461\,{X}^{4}+1861\,{X}^{3}+283\,{X}^{2}+23\,X+1}
\\
\tabB{14}{11110110}{
\left(\frac{2+2b-b^3}{b^2(1+b+b^2)},b\right)}
{ b^7+b^6+2b^5+5b^4+4b^3-4b^2-8b-4}
%
%
{{X}^{14}+37\,{X}^{13}+567\,{X}^{12} +4853\,{X}^{11}+26197\,{X}^{10}
+89695\,{X}^{9} +194611\,{X}^{8}+250446\,{X}^{7}+194611\,{X}^{6}+89695\,{X}^{5}
+26197\,{X}^{4}+4853\,{X}^{3}+567\,{X}^{2}+37\,X+1\\}
\tabB{15}{11110101}{\left(a,a\right)}{a^4+a^3-1}{X^8+117X^7+3671X^6+13396X^5+19683X^4+13396X^3+3671X^2+117X+1}
\\ \\
Seven steps\\
\hline
 \tabB{16}{11101111}{\left(a,1/a\right)}{a^3-a-1}
{{X}^{6}+28\,{X}^{5}+224\,{X}^{4}+345\,{X}^{3}+224\,{X}^{2}+28\,X+1}

\\
\caption{Proving that $G(\cS)$ is infinite.}
\label{tab:ffpt}
\end{longtable}

%

\bigskip
\noindent
{\bf Acknowledgements.} We are grateful to Michael Albert, Jason Bell,
Pierrette Cassou-Nogues, Julie D\'eserti, Arnaud Jehanne and Andrew
Rechnitzer for various discussions and  advice related to this paper.



\bibliographystyle{plain}

\bibliography{qdp}

\end{document}